\documentclass[a4paper,12pt,reqno]{amsart}
\usepackage{fullpage}
\usepackage{mymacros}
\usepackage{tikz}
\usetikzlibrary{arrows.meta,arrows,cd}
\usepackage{pdflscape}
\usepackage[pagebackref]{hyperref}
\usepackage{listings}
\usepackage{oitp}

\title{Compact moduli of marked noncommutative cubic surfaces}

\author{Tarig Abdelgadir}
\address{Mathematical Sciences \\
Loughborough University,
LE11 3TU,
United Kingdom
}
\email{t.abdelgadir@lboro.ac.uk}

\author{Shinnosuke Okawa}
\address{Department of Mathematics, Graduate School of Science,
Osaka University, Machikaneyama 1-1, Toyonaka, Osaka 560-0043, Japan}
\email{okawa@math.sci.osaka-u.ac.jp}

\author{Kazushi Ueda}
\address{Graduate School of Mathematica Sciences,
the University of Tokyo, Komaba 3-8-1, Meguro, Tokyo, 153-8914, Japan.}
\email{kazushi@ms.u-tokyo.ac.jp}

\date{\today}
\begin{document}

%
%-------------------- front matter --------------------
%

\begin{abstract}
We introduce a compact moduli scheme
of marked noncommutative cubic surfaces
as the GIT moduli scheme of relations
of a quiver
associated with a full strong exceptional collection
on a cubic surface.
It is a toric variety
containing the configuration space of six points on a plane
in general position as a locally closed subvariety,
and birationally parametrizes a class of AS-regular $\bZ$-algebras.
\end{abstract}

\maketitle

\tableofcontents

%
%-------------------- Introduction --------------------
%

\section{Introduction}\label{sc:Introduction}

A \emph{cubic surface} can be defined
either as
\begin{enumerate}
\item \label{it:hypersurface}
a smooth hypersurface of degree 3 in $\bP^3$,
\item \label{it:del Pezzo}
a del Pezzo surface of degree 3,
i.e.,
a smooth projective surface whose anti-canonical divisor is ample
of self-intersection number 3,
or
\item \label{it:blow-up}
the blow-up of $\bP^2$ at six points in general position.
\end{enumerate}
% 
% In this paper,
% we 
% % introduce the notion of a noncommutative cubic surface, and
% give yet another characterization of a cubic surface
% as a commutative noncommutative cubic surface,
% by introducing
% an 8-dimensional compact moduli scheme
% of marked noncommutative cubic surfaces,
% containing the configuration space of six points on a plane
% in general position as a locally closed subvariety of codimension 4.
% \footnote{
% The word `noncommutative' is used as a synonym for `not necessarily commutative'
% throughout this paper,
% so that noncommutative objects may or may not be commutative.
% }
% 
Here,
six points on $\bP ^{2}$
is said to be
\emph{in general position}
if
they do not lie on a conic
and no three among them are collinear.
The equivalence of the three definitions
is a well-known classical fact.
A \emph{marking} of a cubic surface
is defined
in such a way that
the moduli scheme
of marked cubic surfaces
is
the quotient
$
\Mcubic \coloneqq (\bP^2)^6_{\mathrm{gen}} / \Aut \bP^2
$
of 
the space
$
(\bP^2)^6_{\mathrm{gen}} \subset (\bP^2)^6
$
of six ordered points on $\bP^2$ in general position
by the free diagonal action of $\Aut \bP^2$.

A \emph{noncommutative projective scheme}
is defined in \cite{Artin-Zhang_NPS}
as a triple
$
X = (\scrC, \cO, s)
$
consisting of an abelian category $\scrC$,
an object $\cO$,
and an autoequivalence $s \colon \scrC \to \scrC$
satisfying certain conditions.
By generalizing the characterizations of a cubic surface,
one can try to
define a noncommutative cubic surface
either as
\begin{enumerate}
\item
a noncommutative hypersurface of degree 3
in a noncommutative $\bP^3$,
\item
a noncommutative projective scheme
$(\scrC, \cO, s)$
such that
$[2] \circ s^{-1}$ is the Serre functor
of the derived category of $\scrC$
and $h^0(s(\cO)) = 4$,
or
\item
the noncommutative blow-up of a noncommutative $\bP^2$
at six points
such that the Serre functor shifted by $-2$
is an ample autoequivalence of the abelian category.
\end{enumerate}
Here, a noncommutative $\bP^n$
is defined as a noncommutative projective scheme
of the form
$\proj S$,
where $S$ is an AS-regular algebra of dimension $n + 1$
whose Hilbert series coincides with that of the polynomial ring \cite{Artin_Schelter}.
By 
its noncommutative hypersurface,
we mean a noncommutative projective scheme
of the form $\proj S/(f)$
for a homogeneous normalizing element $f \in S$,
whose degree is called the degree of the hypersurface.
A noncommutative blow-up is defined in \cite{MR1846352}.
% The relation among these three definitions is unclear.

In this paper,
we give another definition of a noncommutative cubic surface
as
\begin{enumerate}[resume]
\item
the abelian category
$\qgr A$
for
an AS-regular $\tbz$-algebra $A$
of type $\Qt$,
\end{enumerate}
where
an AS-regular $\tbz$-algebra of type $\Qt$ and its
\(
   \qgr
\)
are defined in \pref{sc:AS-regular Z-algebras}.
The starting point is the fact that
any cubic surface $X$ has a full strong exceptional collection
\pref{eq:key collection},
so that one has an equivalence
\begin{align}
D^b \coh X \simeq D^b \module \bfk \Gamma,
\end{align}
where $\Gamma = (Q, I)$
is the quiver $Q$ shown in
\pref{fg:key quiver}
equipped with the two-sided ideal $I \subset \bfk Q$ of relations
such that $\bfk \Gamma \coloneqq \bfk Q / I$
is isomorphic to the total morphism algebra
of the collection.
Since (derived) deformation theories
of an abelian category and its derived category
are identical
(and described by the Hochschild cochain complexes
equipped with the structures of dg Lie algebras),
we can study deformations of
$
D^b  \coh X
$
instead of that of
$
\coh X
$,
which in turn is identical to that of
$
D^b \module \bfk \Gamma
$
by the derived invariance of the Hochschild dg Lie algebra
\cite{keller2003derived}.
The latter is identified with that of
$
\module \bfk \Gamma
$,
and then with that of
$
\bfk \Gamma
$
by the Morita invariance of the deformation theory.
Since deformations of $\bfk \Gamma$ is given by
that of $I$,
we can study noncommutative deformations of $\coh X$
by deformations of $I$.
An AS-regular $\bZ$-algebra $A$
is obtained as (the $\bZ$-algebra associated with)
the zero-th cohomology algebra
of the derived 3-preprojective algebra
(which is also known as the 3-Calabi--Yau completion)
of $\bfk \Gamma$
\cite{Keller_DCYC},
and admits a description
in terms of a quiver with potential
\cite{Ginzburg_CYA}.
The finite-dimensional algebra
$\bfk \Gamma$ is recovered from $A$
as the total morphism algebra
of a full strong exceptional collection
in $D^b \qgr A$.
A finite-dimensional algebra of finite global dimension
whose derived $n$-preprojective algebra
is quasi-isomorphic to its zero-th cohomology
is a \emph{quasi-Fano algebra of dimension $n$}
in the sense of \cite[Definition~1.2]{MR2770441}
(see also \cite{MR2874928}),
and an \emph{$n$-representation infinite algebra}
in the sense of \cite[Definition~2.7]{MR3144232}.

The sets $Q_0$ and $Q_1$
of vertices and arrows
of the quiver $Q$
consist of 9 and 18 elements respectively.
It has 9 relations,
each of which can be regarded as an element of $\bP^2$.
We define the
\emph{compact moduli scheme of marked noncommutative cubic surfaces}
as the geometric invariant theory quotient
$
\Mbar_\chi \coloneqq ((\bP^2)^9)^{\chiss} \GIT (\Gm)^{Q_1}
$
of $(\bP^2)^9$
by the action of the group
$
(\Gm)^{Q_1} \cong (\Gm)^{18}
$
rescaling the arrows,
where
$
\chi \in \Pic^{(\Gm)^{Q_1}} (\bP^2)^9 \cong \bZ^{27}
$
is the stability parameter.
The same idea has been used in \cite{abdelgadir2014compact}
to construct a compact moduli of noncommutative $\bP^2$.
The moduli scheme
$
\Mbar_\chi
$
is a toric variety depending on the stability condition $\chi$,
and the subscheme $M^\circ \subset \Mbar_\chi$
where all components of the homogeneous coordinate are non-zero
is a dense torus
$M^\circ \cong (\Gm)^8$
for any generic
$\chi$ in the secondary fan.

The main result of \cite{2007.07620} shows that
one can construct an AS-regular $\tbz$-algebra
of type $\Qt$
from a decuple
$
\lb
Y, \lb L_v \rb_{v \in Q_0}
\rb
$
consisting of a smooth projective curve $Y$
of genus one
and nine line bundles $(L_v)_{v \in Q_0}$
satisfying the admissibility condition
in \pref{df:admissible}.
This construction gives a rational map
$
\scrA \colon \Mell \dashrightarrow \Mbar_\chi
$
from the fine moduli scheme
$
\Mell
$
of admissible elliptic decuples
given in
\pref{pr:decuple}.

Given a relation $I \subset \bfk Q$ of the quiver $Q$
and 
a generic stability parameter $\theta \in \bZ^{Q_0}$,
the moduli scheme $N_\theta$
of $\theta$-stable representations of $\Gamma = (Q,I)$
of dimension vector $\bsone = (1,\ldots,1) \in \bN^{Q_0}$
is a projective variety
carrying the universal line bundles
$
(\cU_v)_{v \in Q_0}
$.
This construction gives a rational map
$
\scrN_\theta \colon \Mbar_\chi \dashrightarrow \Mell
$.

\begin{theorem} \label{th:birational}
There exists $\theta$
such that
$\scrA$
and
$\scrN_\theta$
are birational inverses of each other.
\end{theorem}

We also prove the following:

\begin{theorem} \label{th:immersion}
By 
sending a marked cubic surface
to the relation of the quiver,
one obtains
a locally closed immersion
$
\iota \colon \Mcubic \to M^\circ
$.
\end{theorem}

It is an interesting problem
to study
the closure of the image of $\iota$ inside $\Mbar_\chi$,
and find a modular interpretation of the boundary.

This paper is organized as follows:
In \pref{sc:marked cubic surfaces},
we recall basic definitions and facts
about marked cubic surfaces,
which can be found, e.g., in \cite{MR2964027}.
In \pref{sc:quiver},
we recall basic definitions on quivers,
$\bZ$-algebras,
and
helices.
In \pref{sc:AS-regular},
we introduce a full strong exceptional collection
of line bundles on cubic surfaces,
a quiver $Q$,
and a class of AS-regular $\bZ$-algebras
which play central roles in this paper.
In \pref{sc:Mrel},
we introduce a toric variety $\Mbar_\chi$
as a compact moduli of relations of the quiver $Q$.
In \pref{sc:N},
we discuss the moduli schemes of stable representations
of the quiver $Q$
with and without relations.
\pref{th:birational} is proved in \pref{sc:Mell},
and \pref{th:immersion} is proved in \pref{sc:immersion},
In \pref{sc:blow-up},
we discuss the relation
with the noncommutative blow-up
in the sense of Van den Bergh.

%
%-------------------- Notations -----------------------
%

\subsection {Notation and Conventions}
%\cofeAm{0.2}{1}{0}{-100pt}{0pt}
Unless otherwise stated,
we work over an algebraically closed field $\bfk$ of characteristic $0$,
so that schemes are defined over $\bfk$,
and categories and functors are linear over $\bfk$.
Varieties are separated integral schemes of finite type,
and surfaces are varieties of dimension two.
% Projective varieties are assumed to be connected
% unless otherwise stated.
We follow the Grothendieck's convention
for the affine space and the projective space
so that
$
\bA V \coloneqq \Spec \Sym V
$
and
$
\bP V = \Proj \Sym V
$
for a $\bfk$-vector space $V$.
The dual space of a $\bfk$-vector $V$
is denoted by $V^\dual$.
Modules are right modules
unless otherwise specified.
For a pair $(X,Y)$ of objects
in a dg category,
the dg vector space of morphisms
and its zero-th cohomology group
is denoted by
$\hom(X,Y)$
and
$\Hom(X,Y)$
respectively.
We also write
$
\Hom^i(X,Y)
\coloneqq
\Hom(X,Y[i])
$
for $i \in \bZ$.

%
%-------------------- Acknowledgements ----------
%

\subsection{Acknowledgements}
S.~O.~was partially supported
by Grants-in-Aid for Scientific Research
(60646909,
16H05994,
16K13746,
16H02141,
16K13743,
16K13755,
16H06337,
19KK0348,
20H01797,
20H01794,
21H04994,
23H01074)
and the Inamori Foundation.
K.~U.~was partially supported
by Grants-in-Aid for Scientific Research
(16K13743,
16H03930,
21K18575).
%
%--------------------------------------------------------
%

\section{Marked cubic surfaces}
\label{sc:marked cubic surfaces}

\subsection{Cubic surfaces}

A \emph{weak del Pezzo surface}
is a connected smooth projective surface
whose anti-canonical divisor is nef and big.
It is a \emph{del Pezzo surface}
if the anti-canonical divisor is ample.
The \emph{degree} of a weak del Pezzo surface
is the self-intersection number of the anti-canonical divisor.
A \emph{cubic surface}
is a del Pezzo surface of degree 3.
The anti-canonical bundle
$
\omega _{X} ^{-1}
$
of a cubic surface $X$
is very ample with
$
h ^{0} \left( X, \omega _{X} ^{-1}\right) = 4
$,
so that
\begin{enumerate}
\item
the anti-canonical embedding
\begin{align}
  \Phi _{ \left| \omega _{X} ^{-1} \right|} \colon X \hookrightarrow \bP H ^{0} \left( X, \omega _{X} ^{-1} \right) \simeq \bP ^{3}
\end{align}
realizes \(X\) as a cubic hypersurface of
\(
   \bP ^{ 3 }
\),
and
\item
a pair of cubic surfaces are isomorphic to each other if and only if they are projectively equivalent.
\end{enumerate}
A curve on a cubic surface
is said to be a \emph{line}
if it has self-intersection number $-1$.
They are exactly the curves
which goes to a line
in the anti-canonical embedding.
Any cubic surface contains
exactly 27 lines.
Any six mutually disjoint lines on a cubic surface $X$
can be contracted to obtain $\bP^2$,
and the images of the six lines are in general position on $\bP^2$.
Conversely,
the blow-up of $\bP^2$
at any six points in general position
is a cubic surface.

\subsection{Markings}

A \emph{geometric marking} of a cubic surface $X$
is a sequence $(l_i)_{i=1}^6$
of six lines on $X$
which are mutually disjoint.
% i.e., $l_i \cap l_j = \emptyset$ for any $1 \le i < j \le 6$.
A \emph{geometrically marked cubic surface}
is a pair
$
(X,(l_i)_{i=1}^6)
$
of a cubic surface $X$
and a geometric marking.
An \emph{isomorphism}
of a pair
$
(
(X,(l_i)_{i=1}^6),
(X',(l_i')_{i=1}^6)
)
$
of geometrically marked cubic surfaces
is an isomorphism
$
\phi \colon X \to X'
$
of cubic surfaces
such that
$\phi(l_i) = l_i'$
for any $1 \le i \le 6$.

Let $\LL$ be the lattice
generated by $\{ \bfe_i \}_{i=0}^6$
as a free abelian group
equipped with the symmetric bilinear form
\begin{align}
\bfe_i \cdot \bfe_j =
\begin{cases}
1 & i=j=0, \\
-1 & 1 \le i = j \le 6, \\
0 & \text{otherwise}.
\end{cases}
\end{align}
The orthogonal lattice
$\delta^\bot$
of the element
\begin{align}
\delta \coloneqq 3 \bfe_0 - \bfe_1 - \cdots - \bfe_6 \in \LL
\end{align}
can be identified with the root lattice $\sfE_6$
of the
negative definite
root system of type $E_6$.

A \emph{lattice-theoretic marking}
of a cubic surface $X$
is an isometry
$
\varphi \colon \LL \simto \Pic X
$
sending $\delta$ to \(\omega ^{ - 1 } _{ X }\).
A \emph{lattice-theoretically marked cubic surface}
is a pair $(X, \varphi)$
of a cubic surface
and a lattice-theoretic marking.
An \emph{isomorphism}
of a pair
$
(
(X,\varphi),
(X',\varphi')
)
$
of lattice-theoretically marked cubic surfaces
is an isomorphism
$
\phi \colon X \to X'
$
of cubic surfaces
such that
$\varphi = \phi^* \circ \varphi'$.

Since the notions of a geometric marking
and a lattice-theoretic marking are equivalent,
we will use them interchangeably
and refer to them simply as a \emph{marking}.

\subsection{Moduli}

Let
$
(\bP^2)^6_{\mathrm{gen}} \subset (\bP^2)^6
$
be the configuration space
of six points on $\bP^2$ in general position.
The quotient scheme
$
\cMcubic \coloneqq (\bP^2)^6_{\mathrm{gen}} / \Aut \bP^2
$
is the moduli scheme
of marked cubic surfaces,
and the quotient stack
$
\ld \cMcubic \middle/ W(\sfE_6) \rd
$
is the moduli stack
of cubic surfaces,
where $W(\sfE_6) \cong \rO(\sfE_6)$
is the Weyl group of type $\sfE_6$.

While moduli of (marked) cubic surfaces
is a classical subject
which goes back to the nineteenth century,
% See,
% e.g.,
% \cite{MR2964027}
% and references therein
% for classical works.
they are still studied today
from various points of view,
such as
uniformizations
% by hypergeometric functions
and compactifications.
See,
e.g.,
\cite{
MR662660,
MR1910264,
% MR2583265,
% MR2789835,
MR2196731,
MR2534095}
and references therein
for some of recent results.

\section{Quivers, $\bZ$-algebras, and helices}
\label{sc:quiver}

\subsection{Quivers and path algebras}

A \emph{quiver}
$
Q = (Q_0, Q_1, s, t)
$
consists of
\begin{itemize}
 \item
a set $Q_0$ of vertices,
 \item
a set $Q_1$ of arrows, and
 \item
two maps
$s$ and $t$
from $Q_1$ to $Q_0$.
\end{itemize}
For an arrow $a \in Q_1$,
the vertices $s(a)$ and $t(a)$
are
called the \emph{source}
and the \emph{target}
of the arrow $a$
respectively.
A \emph{path}
is either
\begin{itemize}
\item
a path of length $k$
for some positive integer $k$,
which is a sequence
$p = (a_k, a_{k-1}, \dots, a_{1}) \in (Q_1)^k$
of arrows
such that $s(a_{i+1}) = t(a_i)$
for all $i \in \{ 1, \dots, k-1 \}$,
in which case we define the \emph{source} and the \emph{target} of $p$
by $s(p) \coloneqq s(a_1)$
and $t(p) \coloneqq t(a_k)$,
or 
\item
a path $\bfe_v$ of length zero
for some $v \in Q_0$,
in which case we set
$
s(\bfe_v)
= t(\bfe_v)
= v
$.
\end{itemize}
The \emph{concatenation} of paths
is defined by
\begin{align}
 (b_m, \dots, b_1) \cdot (a_n, \dots, a_1)
  = \begin{cases}
     (b_m, \dots, b_1, a_n, \dots, a_1) & s(b_1) = t(a_n), \\
      0 & \text{otherwise}
    \end{cases}
\end{align}
for paths of positive length,
and
\begin{align}
p \cdot \bfe_v
&=
\begin{cases}
p & s(p) = v, \\
0 & \text{otherwise},
\end{cases}
\\
\bfe_v \cdot p
&=
\begin{cases}
p & t(p) = v, \\
0 & \text{otherwise}
\end{cases}
\end{align}
if the length of at least one of the paths is zero.
For a non-negative integer $k$,
let
$\scrP_k(Q)$
be the set of paths of length $k$.
The \emph{path algebra} $\bfk Q$
is the algebra
freely generated
by the set
$
\scrP(Q)
\coloneqq
\coprod_{k=0}^\infty \scrP_k(Q)
$
of paths
as a vector space,
and the multiplication is defined
by the concatenation of paths.
It is graded by the length of the path as
\begin{align}
\bfk Q = \bigoplus_{k=0}^\infty \lb \bfk Q \rb_k,
\end{align}
where
$
\lb \bfk Q \rb_k
\coloneqq \bfk \scrP_k(Q)
$
is the vector space
freely generated by $\scrP_k(Q)$.
A \emph{quiver with relations} is a pair
$
\Gamma = (Q, I)
$
consisting of a quiver $Q$
and a two-sided ideal $I$
of the path algebra $\bfk Q$.
The \emph{path algebra} of $\Gamma$ is defined as
\begin{align}
\bfk \Gamma \coloneqq \bfk Q / I.
\end{align}
% A \emph{representation}
% $M$
% of a quiver is a right module
% over the path algebra $A$,
% whose \emph{dimension vector}
% is defined by
% \begin{align}
% \dim M
% \coloneqq
% \lb \dim_\bfk \lb M \bfe_v \rb \rb_{v \in Q_0}
% \in
% ( \bN \cup \{ \infty \} )^{Q_0}.
% \end{align}

\subsection{$J$-algebras for a set $J$}

Let $J$ be a set.
A \emph{$J$-algebra} over
\(
   \bfk
\)
is a
\(
   \bfk
\)-linear category whose set of objects is $J$.
A $J$-algebra can equivalently be defined as an algebra over the semisimple algebra
$
\bfk^J \coloneqq \prod_{J} \bfk.
$
The path algebra $\bfk \Gamma$
is naturally a $Q_0$-algebra,
and any $Q_0$-algebra
can be described by a quiver with relations
with $Q_0$ as the set of vertices.

\subsection{AS-regular $\bZ$-algebras}

A $\bZ$-algebra $A$
consists of vector spaces
$A_{ji} \coloneqq \Hom_A(i,j)$
% of morphisms
for $i,j \in \bZ$,
the identify element
$
\bfe_i
$
for $i \in \bZ$,
and linear maps
$A_{kj} \otimes A_{ji} \to A_{ki}$
satisfying a set of axioms.
A module $M$ over a $\bZ$-algebra $A$
is a functor from $A$ to the category of vector spaces,
i.e.,
a collection of vector spaces
$M_i$ for $i \in \bZ$
and linear maps
$M_k \otimes A_{ki} \to M_i$
for $i,k \in \bZ$
satisfying the obvious axioms.
The category $\Module A$
of $A$-modules is a Grothendieck category.
A $\bZ$-algebra $A$ is said to be \emph{positively graded}
if $A_{ji} = 0$ for any $j < i$.
A positively graded $\bZ$-algebra is said to be \emph{connected}
if $A_{ii} = \bfk \bfe_i$ for any $i \in \bZ$.
For a connected $\bZ$-algebra $A$
and an integer $i$,
let $P_i \coloneqq \bfe_i A$
(resp.~$S_i \coloneqq \bfe_i A \bfe_i$)
be the $i$-th projective (resp.~simple)
$A$-module.
One has
$
(P_i)_j = A_{ij}
$
(resp.~$(S_i)_i = \bfk$ and $(S_i)_j = 0$
for any $j \ne i$).
For a non-negative integer $d$,
a connected $\bZ$-algebra $A$ is said to be
\emph{AS-Gorenstein} of dimension $d$ if for any
\(
   i
   \in
   \bZ
\)
there exists the unique
\(
   j _{ 0 }
   \in
   \bZ
\)
as follows.
\begin{align}
    \Ext^k(S_i, P_j) \cong
    \begin{cases}
    \bfk & j = j _{ 0 } \text{ and } k = d, \\
    0 & \text{otherwise},
    \end{cases}
\end{align}
where `AS' stands for Artin--Schelter
\cite{Artin_Schelter}.
An AS-Gorenstein $\bZ$-algebra is said to be \emph{AS-regular}
if $\Module A$ has finite global dimension and
\(
   A
\)
has finite GK dimension; namely, there exits a polynomial
\(
   p ( t )
   \in
   \bQ [ t ]
\)
such that
\(
   \dim
   A
   _{
    i j
   }
   <
   p ( i - j )
\)
for any pair
\(
   ( i, j )
   \in
   \bZ ^{ 2 }
\).

%
%------------------------------------------------------------------------
%
\subsection{Potentials}

For a non-negative integer $k$,
consider the action of the cyclic group of order $k$
on
$\lb Q_1 \rb^k$
generated by
$
(a_k,a_{k-1},\ldots,a_1)
\mapsto
(a_1,a_k,\ldots,a_2).
$
A \emph{cyclic path} of length $k$
is an element
$
c = [a_k,\ldots,a_1] \in \lb Q_1 \rb^k/(\bZ/k\bZ)
$
such that $(a_k, \ldots, a_1) \in \scrP(Q)$ and
$
s(a_1) = t(a_k)
$.
Let $\scrC_k(Q)$ be the set of cyclic paths of length $k$,
% We also set $\scrC_0(Q) \coloneqq \lc \bfe_v \rc_{v \in Q_0}$.
and $\bfk \scrC_k(Q)$ the vector space
freely generated by $\scrC_k(Q)$.
For each $a \in Q_1$,
we define a map
\begin{align}
\frac{\partial}{\partial a} \colon \bfk \scrC_k(Q) \to \lb \bfk Q \rb_{k-1}
\end{align}
by sending a basis $[a_k,\ldots,a_1]$ to
\begin{align}
  \sum_{i=1}^k \delta_{a_i,a} a_{i-1} \cdots a_1 a_k \cdots a_{i+1}
\end{align}
and extending linearly,
where
\begin{align}
  \delta_{a,a'} =
\begin{cases}
  1 & a = a', \\
  0 & \text{otherwise}
\end{cases}
\end{align}
is the Kronecker delta.
A \emph{quiver with potential}
is a pair $(Q, \Phi)$
consisting of a quiver $Q$
and an element
$
\Phi \in \bfk \scrC(Q) \coloneqq \bigoplus_{k=0}^\infty \bfk \scrC_k(Q)
$
called a \emph{potential}.
A quiver with potential $(Q, \Phi)$
gives rise to a quiver with relations $\Gamma = (Q, \partial \Phi)$
where
\begin{align}
    \partial \Phi \coloneqq \lb \frac{\partial}{\partial a} \Phi \rb_{a \in Q_1}.
\end{align}

%
%------------------------------------------------------------------------
%

\subsection{Exceptional collections and helices}
 \label{sc:Exceptional collection and helices}

A dg category is \emph{proper}
if
$\hom(X,Y)$ is a perfect dg vector space
(i.e., quasi-equivalent to a bounded complex
of finite-dimensional vector spaces)
for any objects $X, Y$.
An object $E$ of a proper dg category $\scrD$ is \emph{exceptional}
if
\begin{align}
 \hom(E, E) \simeq \bfk.
\end{align}
A sequence
$
(E_1, \ldots, E_\ell)
$
of exceptional objects
is an \emph{exceptional collection} if
\begin{align}
 \hom(E_j, E_i) \simeq 0
\end{align}
for any $1 \le i < j \le \ell$.
An exceptional collection
$
(E_1, \ldots, E_\ell)
$
is \emph{strong} if
\begin{align}
 \Hom^i(E_j, E_k) = 0
\end{align}
for any $i \ne 0$
and any $j, k \in \{ 1, \ldots, \ell \}$.
An exceptional collection
$
(E_1, \ldots, E_\ell)
$
is \emph{full}
if $\scrD$ is equivalent to
the pretriangulated hull
(i.e., the smallest full subcategory
closed under shifts and cones)
of $\{ E_1, \ldots, E_\ell \}$.

If
$
(E_1, \ldots, E_\ell)
$
is a full strong exceptional collection in $\scrD$,
then one has a quasi-equivalence
\begin{align}
 \hom_{\scrD} \lb \bigoplus_{i=1}^{\ell} E_i , - \rb
 \colon
 \scrD \simto D^b \module A
\end{align}
of dg categories
between $\scrD$ and the dg category of bounded complexes
of finitely-generated right modules
over the total morphism algebra
$
 A = \bigoplus_{i,j=1}^{ \ell } \Hom(E_i, E_j)
$
\cite{Bondal_RAACS,Rickard,Keller_DDC}.

For $1 \le i \le \ell-1$,
the \emph{left mutation}
of an exceptional collection
$
 \tau = \lb E_j \rb_{j=1}^\ell
$
at position $i$
is defined by
\begin{align}
 L_i (\tau) \coloneqq
  \lb E_1, \ldots, E_{i-1}, L_{E_i} E_{i+1}, E_i, E_{i+2}, \ldots, E_\ell \rb
\end{align}
where
\begin{align}
 L_{E_i} E_{i+1} \coloneqq \Cone \lb \hom(E_i, E_{i+1}) \otimes E_i \xto{\ev} E_{i+1} \rb.
\end{align}
It is inverse
to the \emph{right mutation}
defined by
\begin{align}
 R_i (\tau) \coloneqq
  ( E_1, \ldots, E_{i-1}, E_{i+1}, R_{E_{i+1}} E_i, E_{i+2}, \ldots, E_\ell )
\end{align}
where
\begin{align}
  R_{E_{i+1}} E_i \coloneqq \Cone \lb E_i \xto{\coev} \hom(E_i, E_j)^\dual \otimes E_j \rb.
\end{align}
The \emph{helix}
generated by an exceptional collection $\lb E_i \rb_{i=1}^\ell$
is the sequence $\lb E_i \rb_{i \in \bZ}$
of exceptional objects
satisfying
\begin{align} \label{eq:helix}
 E_{i-\ell} = L_{E_{i-\ell+1}} \circ L_{E_{i-\ell+2}}
  \circ \cdots \circ L_{E_{i-1}}(E_i)[ - 2 ]
\end{align}
for any $i \in \bZ$.
The shift is chosen in such a way that
$E_{i-\ell} \cong \bS(E_i)[-2]$,
so that
if
$\scrD \simeq D^b \coh X$
for a smooth projective surface $X$,
then one has
$E_{i-\ell} \cong E_i \otimes \omega_X$.
The length $\ell$ of the exceptional collection
generating the helix is called the \emph{period} of the helix.
For any $i \in \bZ$,
the exceptional collection $\lb E_j \rb_{j=i}^{i+\ell-1}$
consisting of $\ell$ consecutive objects in the helix
is called a \emph{foundation}
of the helix.
A helix
$
\lb E_i \rb_{i \in \bZ}
$
is said to be \emph{strong}
if every foundation is strong.
It is said to be \emph{acyclic} if it satisfies a stronger condition that
$
 \Hom^{k}(E_i, E_j) = 0
$
for any $-\infty < i < j < \infty$
and any $k \ne 0$.
Acyclicity is closely related to,
but different from,
the \emph{geometricity}
in \cite{MR1230966};
% (which is called \emph{simplicity} in \cite{MR2142382})
since the shift in \eqref{eq:helix} is absent in \cite{MR1230966},
geometricity and acyclicity coincides
if and only if $\ell=3$.
An acyclic helix
$(E_i)_{i \in \bZ}$
gives the $\bZ$-algebra $A$
defined by
$
A_{ji} = \Hom(E_i, E_j),
$
which is AS-regular of dimension 3
as explained in \cite{2007.07620}.

\section{AS-regular $\tbz$-algebras}
\label{sc:AS-regular}

\subsection{An exceptional collection}

\begin{figure}
\begin{tikzpicture}
[
scale=2.5,
inner sep=.5mm,
node distance=10mm,
every node/.style={fill=white},
arrows = {-Stealth[length=2mm]}
]
    \path node (trivial) at ( 0, 0 ) [shape = circle, draw] {$0, 0$}
          node (e1 + e2 + e3 - h) at ( 0, - 1 ) [shape = circle, draw] {$1, 0$}
          node (sum e1 to e6 - 2h) at ( 0, - 2 ) [shape = circle, draw] {$2, 0$}
          node (e1) at ( 1, 0 ) [shape = circle, draw] {$0, 1$}
          node (e2) at ( 1, - 1 ) [shape = circle, draw] {$1, 1$}
          node (e3) at ( 1, - 2 ) [shape = circle, draw] {$2, 1$}
          node (h - e4) at ( 2, 0 ) [shape = circle, draw] {$0, 2$}
          node (h - e5) at ( 2, - 1 ) [shape = circle, draw] {$1, 2$}
          node (h - e6) at ( 2, - 2 ) [shape = circle, draw] {$2, 2$};
    
    %column 0 to column 1
    \draw
    (trivial) edge (e1)
    (trivial) edge (e2)
    (trivial) edge (e3)

    (e1 + e2 + e3 - h) edge (e2)
    (e1 + e2 + e3 - h) edge (e3)
    (e1 + e2 + e3 - h) edge (e1)

    (sum e1 to e6 - 2h) edge (e3)
    (sum e1 to e6 - 2h) edge (e1)
    (sum e1 to e6 - 2h) edge (e2)

    %column 1 to column 2

    (e1) edge (h - e4)
    (e1) edge (h - e5)
    (e1) edge (h - e6)

    (e2) edge (h - e5)
    (e2) edge (h - e6)
    (e2) edge (h - e4)

    (e3) edge (h - e6)
    (e3) edge (h - e4)
    (e3) edge (h - e5);
\end{tikzpicture}
\caption{The quiver $Q$}
\label{fg:key quiver}
\end{figure}
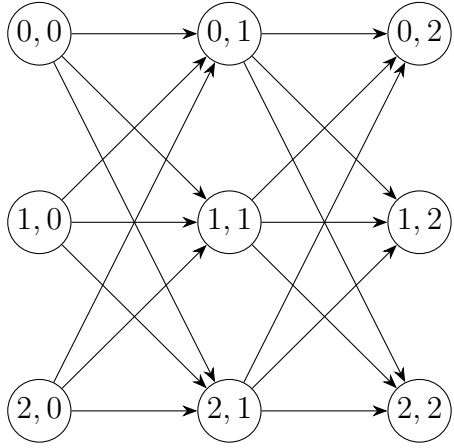

A cubic surface $X$ has
a full strong exceptional collection
consisting of nine line bundles
\begin{align} \label{eq:key collection}
\begin{aligned}
E_{0,0} &\coloneqq \cO_X,
&
E_{0,1} &\coloneqq\cO_X \lb l_1 \rb,
&
E_{0,2} &\coloneqq \cO_X \lb l-l_4 \rb,
\\
E_{1,0} &\coloneqq \cO_X \lb -2l + \textstyle{ \sum_{ i= 1 }^6 l_i } \rb,
&
E_{1,1} &\coloneqq \cO_X \lb l_2 \rb,
&
E_{1,2} &\coloneqq \cO_X \lb l-l_5 \rb,
\\
E_{2,0} &\coloneqq \cO_X \lb -l+l_1+l_2+l_3 \rb,
&
E_{2,1} &\coloneqq \cO_X \lb l_3 \rb,
&
E_{2,2} &\coloneqq\cO_X \lb l-l_6 \rb,
\end{aligned}
\end{align}
where $l$ is the strict transform of the hyperplane in $\bP^2$
and $l_i$ for $1 \le i \le 6$ are exceptional divisors.
This is a \emph{three-block collection}
in the sense that the collection is divided into three \emph{blocks}
such that $\Ext^k(X,Y) = 0$ for any $k \in \bZ$
if $X$ and $Y$ in the same block.
The standard definition of an exceptional collection
recalled in
\pref{sc:Exceptional collection and helices}
requires a choice of a total order,
which is unique only up to permutations
within each block.
We avoid making this unnatural choice,
and use $\tbt$ as the index set for an exceptional collection,
and $\tbz$ as that for the associated helix.
We freely identify $\{ 0, 1, 2 \}$
with $\bZ/3\bZ$
by the inclusion
$
\{ 0, 1, 2 \} \hookrightarrow \bZ
$
composed with the projection
$
\bZ \to \bZ/3\bZ
$.
The identification of the index set $\tbt$
for the exceptional collection \eqref{eq:key collection}
on a cubic surface
with $(\bZ/3\bZ)^2$
is natural from the point of view of McKay correspondence
\cite{Kapranov-Vasserot,Bridgeland-King-Reid}
as explained in 
\cite[Section 7]{Ueda_SMQCCS}
(cf.~also
\cite{Ueda-Yamazaki_NBTMQ,
MR3100950}).
The three-block collection
\pref{eq:key collection}
is related
to the three-block collection
$\tau_{(6.1)}$
in \cite[p.~453]{MR1642152}
by a choice of a foundation of the helix.
The helix $(E_{\bsi})_{\bsi \in \tbz}$
generated by $(E_{\bsi})_{\tbt}$
is acyclic,
so that the resulting
$\tbz$-algebra $A$ is AS-regular of dimension 3.

\subsection{A quiver with relations}

The total morphism algebra
of the collection \pref{eq:key collection}
is described by
the quiver
$
Q = (Q_0, Q_1, s, t)
$
shown in
\pref{fg:key quiver}.
We identify a vertex
$
(i,j) \in Q_0 \coloneqq \tbt
$
with the exceptional object $E_{i,j}$
in the collection \pref{eq:key collection}.
For any
$
(i,j,k) \in \bZ/3\bZ \times \{ 0,1 \} \times \bZ/3\bZ
$,
the space
$\Hom(E_{i,j}, E_{i+k,j+1})$
is one-dimensional 
since six points at the center of the blow-up
are in general position.
We choose a basis
$
x_{i,j,k}
$
for each of these spaces,
and identify it with an arrow of $Q$.
This induces a surjection
$
\bfk Q
\to
\End \lb \bigoplus_{v \in Q_0} E_v \rb
$
whose kernel $I$ satisfies
\begin{align} \label{eq:dimension matrix}
  \dim \bfe_{i',j'} I \bfe_{i,j} =
  \begin{cases}
    1 & j = 0 \text{ and } j' = 2, \\
    0 & \text{otherwise}.
  \end{cases} 
\end{align}

\begin{figure}[t]
\begin{tikzpicture}
[
scale=2.5,
inner sep=.5mm,
node distance=10mm,
every node/.style={fill=white},
arrows = {-Stealth[length=2mm]}
]

    \path node (trivial) at ( 0, 0 ) [shape = circle, draw] {0, 0}
          node (e1 + e2 + e3 - h) at ( 0, - 1 ) [shape = circle, draw] {1, 0}
          node (sum e1 to e6 - 2h) at ( 0, - 2 ) [shape = circle, draw] {2, 0}
          node (e1) at ( 1, 0 ) [shape = circle, draw] {0, 1}
          node (e2) at ( 1, - 1 ) [shape = circle, draw] {1, 1}
          node (e3) at ( 1, - 2 ) [shape = circle, draw] {2, 1}
          node (h - e4) at ( 2, 0 ) [shape = circle, draw] {0, 2}
          node (h - e5) at ( 2, - 1 ) [shape = circle, draw] {1, 2}
          node (h - e6) at ( 2, - 2 ) [shape = circle, draw] {2, 2}
          node (3h - sum e1 to e6) at ( 3, 0 ) [shape = circle, draw] {0, 0}
          node (2h - e4 - e5 - e6) at ( 3, - 1 ) [shape = circle, draw] {1, 0}
          node (h) at ( 3, - 2 ) [shape = circle, draw] {2, 0};
    
    %column 0 to column 1
    \draw
    (trivial) edge (e1)
    (trivial) edge (e2)
    (trivial) edge (e3)

    (e1 + e2 + e3 - h) edge (e2)
    (e1 + e2 + e3 - h) edge (e3)
    (e1 + e2 + e3 - h) edge (e1)

    (sum e1 to e6 - 2h) edge (e3)
    (sum e1 to e6 - 2h) edge (e1)
    (sum e1 to e6 - 2h) edge (e2)

    %column 1 to column 2

    (e1) edge (h - e4)
    (e1) edge (h - e5)
    (e1) edge (h - e6)

    (e2) edge (h - e5)
    (e2) edge (h - e6)
    (e2) edge (h - e4)

    (e3) edge (h - e6)
    (e3) edge (h - e4)
    (e3) edge (h - e5)

    %column 2 to column 3
    (h - e4) edge (3h - sum e1 to e6)
    (h - e4) edge (2h - e4 - e5 - e6)
    (h - e4) edge (h)
    
    (h - e5) edge (2h - e4 - e5 - e6)
    (h - e5) edge (h)
    (h - e5) edge (3h - sum e1 to e6)
    
    (h - e6) edge (h)
    (h - e6) edge (3h - sum e1 to e6)
    (h - e6) edge (2h - e4 - e5 - e6);
\end{tikzpicture}
\caption{The rolled-up quiver $\Qt$}
\label{fg:rolled-up quiver}
\end{figure}

\subsection{The rolled-up quiver}

A non-zero element $b \in \bfk Q$ is said to be \emph{homogeneous}
if there exist
$v, v' \in Q_0$
such that $b = \bfe_{v'} b \bfe_v$,
in which case
we write $s(b) \coloneqq v$ and $t(b) \coloneqq v'$.
Fix a homogeneous basis $\bfB$ of $I$.
In addition,
we fix a finite set $B$ and a bijection $r \colon B \to \bfB$
to avoid confusion in the following construction.
The \emph{rolled-up quiver}
$
\Qt = (\Qt_0, \Qt_1, \tilde{s}, \tilde{t})
$
associated with $\Gamma = (Q, I)$
is the quiver
shown in \pref{fg:rolled-up quiver},
which is defined by
\begin{align}
\Qt_0 &\coloneqq Q_0, \\
\Qt_1 &\coloneqq Q_1 \sqcup B, \\
\tilde{s}(a) &=
\begin{cases}
  s(a) & a \in Q_1, \\
  t(r(a)) & a \in B,
\end{cases}\\
\tilde{t}(a) &=
\begin{cases}
  t(a) & a \in Q_1, \\
  s(r(a)) & a \in B.
\end{cases}
\end{align}

\subsection{Potentials}

The sum
\begin{align} \label{eq:potential from relation}
  \Phi \coloneqq \sum_{b \in B} [b \cdot r(b)] \in \bfk \scrC_3 \lb \Qt \rb
\end{align}
is called the \emph{potential} associated with $I$.
It is natural to set
\begin{align}
B = \{ x_{i,2,k} \}_{i,k=0}^2, \qquad
\tilde{s}(x_{i,2,k}) = (i,2), \qquad
\tilde{t}(x_{i,2,k}) = (i+k,0).
\end{align}

\subsection{Calabi--Yau algebras of dimension 3}

The path algebra of the quiver with relations
$(\Qt, \partial \Phi)$
is isomorphic to the rolled-up helix algebra
\begin{align} \label{eq:rolled-up helix algebra}
B \coloneqq \bigoplus_{v,v' \in Q_0} \bigoplus_{k=0}^\infty
\Hom \lb E_v, E_{v'} \otimes \omega^{-k} \rb,
\end{align}
which is a Calabi--Yau algebra of dimension 3
in the sense of \cite[Definition 3.2.3]{Ginzburg_CYA}.
The rolled-up helix algebra $B$ is a $\bZ$-graded
$\tbt$-algebra,
and $A$ is the $\tbz$-algebra
associated with
% the $\tbz$-graded algebra 
$B$.
It follows from \cite[Theorem 5.3.1]{Ginzburg_CYA}
(see also \cite[Proposition 5.1.9]{Ginzburg_CYA})
that $B$ as the diagonal bimodule
has a resolution of the form
\begin{multline} \label{eq:bimodule resolution}
0
\to
(B \otimes B)^{\bfk^{Q_0}}
\to
B \otimes_{\bfk^{Q_0}} \lb \bfk \Qt_1 \rb^\dual \otimes_{\bfk^{Q_0}} B
\to
B \otimes_{\bfk^{Q_0}} \bfk \Qt_1 \otimes_{\bfk^{Q_0}} B\\
\to
B \otimes_{\bfk^{Q_0}} B
\to
B
\to
0
\end{multline}
where
\begin{align}
(B \otimes B)^{\bfk^{Q_0}}
\coloneqq
\lc
b \in B \otimes B \relmid
r b = b r \text{ for any } r \in \bfk^{Q_0}
\rc
\end{align}
and
$\bfk \Qt_1$ is the vector space
freely generated by $\Qt_1$
equipped with the natural structure
of a $\bfk^{Q_0}$-bimodule.
By tensoring the simple $B$-module
$
S_v \coloneqq \bfe_v B \bfe_v
$
supported on $v \in Q_0$
from the left to \pref{eq:bimodule resolution},
one obtains the resolution
\begin{align} \label{eq:resolution of simple}
0
\to
P_v
\to
\bigoplus_{t(a)=v} P_{s(a)}
\to
\bigoplus_{s(a)=v} P_{t(a)}
\to
P_v
\to
S_v
\to
0
\end{align}
of the simple $B$-module
in terms of the projective modules $P_v \coloneqq \bfe_v B$.

\subsection{AS-regular $\tbz$-algebras}
\label{sc:AS-regular Z-algebras}
\pref{eq:resolution of simple} induces the minimal resolution
\begin{align}\label{eq:minimal resolution}
  0
  \to
  P _{ \bi + ( 0, 3 ) }
  \to
  \bigoplus _{ a \in \bZ / 3 \bZ } P _{ \bi + ( a, 2 )}
  \to
  \bigoplus _{ b \in \bZ / 3 \bZ } P _{ \bi + ( b, 1 )}
  \to
  P _{ \bi }
  \to
  S _{ \bi }
  \to
  0
\end{align}
of the simple $A$-module.

\begin{definition} \label{df:AS-regular of type Qt}
An AS-regular $\tbz$-algebra of dimension 3
is said to be \emph{of type $\Qt$}
if the minimal resolution of the simple module $S_\bsi$
has the form \eqref{eq:minimal resolution}
for any $\bsi \in \tbz$.
\end{definition}

Let
$\module A$ be the full subcategory of $\Module A$
consisting of Noetherian objects.
Let further $\tor A$ be the full subcategory of $\module A$
consisting of $A$-modules $M$ such that $M_{(i,j)} = 0$
if $j \gg 0$,
and $\Tor A$ be the full subcategory of $\Module A$
consisting of colimits of objects in $\tor A$.
Then $\tor A$ is a Serre subcategory of $\module A$,
$\Tor A$ is a localizing subcategory of $\Module A$,
and one has equivalences
$\Qgr A \coloneqq \Module A / \Tor A \cong \Qcoh X$
and
$\qgr A \coloneqq \module A / \tor A \cong \coh X$.

\section{Moduli of relations}
\label{sc:Mrel}

\subsection{Moduli of relations}

Note that
any linear subspace
$
I \subset \bfk Q
$
satisfying \eqref{eq:dimension matrix}
is a 2-sided ideal.
The group
\begin{align}\label{eq:gauge group}
  G \coloneqq \Aut _{\Alg \left( Q_0 \right)} \left( \bfk Q \right)
  =
  \prod _{ a \in Q _{ 1 }} \GL \left( \bfk a \right)
  = \lb \Gm \rb^{Q_1}
  \cong
  \left( \Gm \right) ^{ 18 }
\end{align}
of automorphisms
of the $Q_0$-algebra $\bfk Q$
acts naturally on
the space
\begin{align}\label{eq:space of ideals}
  \cI \coloneqq
  \left\{ I \subset \bfk Q \mid \mbox{\(I\) satisfies \eqref{eq:dimension matrix}}\right\}
  =
  \prod_{i,j = 0}^2 \bP \left( \bfe_{(j,2)} \lb \bfk Q \rb \bfe_{(i,0)} \right)^\dual
  \cong
  \left( \bP ^{2} \right) ^{ 9 }
\end{align}
of two-sided ideals $I$
satisfying \eqref{eq:dimension matrix},
and
the \emph{moduli stack of relations}
% of the quiver
% \pref{fg:The quiver corresponding to the key full exceptional collection}
% satisfying \pref{eq:dimension matrix}
is defined as the quotient stack
\begin{align} \label{eq:moduli stack of marked nc cubics}
  \cMrel \coloneqq \left[ \cI / G \right].
\end{align}
Note that the scheme
$\cI$
is a GIT quotient
of $\bA^{27}$
by an action of $\lb \Gm \rb^9$,
so that a GIT quotient
of $\cI$ by $G$
is a GIT quotient
of $\bA^{27}$
by an action of $\lb \Gm \rb^{27}$.
The latter has a natural interpretation
in terms of the rolled-up quiver.

\subsection{Moduli of potentials}

The group
\begin{align}\label{eq:gauge group2}
\lb \Gm \rb^{\Qt_1}
= \Aut _{\Alg \left( Q_0 \right)} \left( \bfk \Qt \right)
= \prod _{ a \in \Qt_1} \GL \left( \bfk a \right)
\cong \lb \Gm \rb^{27}
\end{align}
acts naturally on the affine space
\begin{align}
\bA^{\scrC_3 \lb \Qt \rb}
% = \lb \bA^1 \rb^{\scrC_3 \lb \Qt \rb}
% \cong
% \bA \lb \bfk \scrC_3 \lb \Qt \rb^\dual \rb
\coloneqq
\Spec \Sym \lb \bfk \scrC_3 \lb \Qt \rb^\dual \rb
\end{align}
through the torus
$
\lb \Gm \rb^{\scrC_3 \lb \Qt \rb}
% = \lb \Gm \rb^{\scrC_3 \lb \Qt \rb} \subset \GL \lb \bfk \scrC_3 \lb \Qt \rb \rb.
$.
One has
\begin{align}
\Qt_1
=
\lc (i, j) \xto{x_{i,j,k}} (i+k,j+1) \rc_{i,j,k=0}^2
\end{align}
and
\begin{align}
\scrC_3 \lb \Qt \rb
=
\lc x_{i+j+k,2,-j-k} x_{i+j,1,k} x_{i,0,j} \rc_{i,j,k=0}^2,
\end{align}
so that the homomorphism
\begin{align}
\rho \colon \lb \Gm \rb^{\Qt_1}
\to
\lb \Gm \rb^{\scrC_3 \lb \Qt \rb}.
\end{align}
corresponding to the action
\begin{align}
\lb \Gm \rb^{\Qt_1}
\curvearrowright
\lb \bA \rb^{\scrC_3 \lb \Qt \rb}
\end{align}
is given by
\begin{align}
\rho \lb \lb \alpha_{i,j,k} \rb_{i,j,k=0}^2 \rb
=
\lb \alpha_{i+j+k,2,-j-k} \alpha_{i+j,1,k} \alpha_{i,0,j} \rb_{i,j,k=0}^2.
\end{align}
The quotient stack
\begin{align}\label{eq:moduli stack of potentials}
\cMpot
\coloneqq
\left[
\bA^{\scrC_3 \lb \Qt \rb}
\middle/
\lb \Gm \rb^{\Qt_1}
\right]
\end{align}
is called the \emph{moduli stack of potentials}.

\subsection{An open immersion}

\begin{proposition}\label{pr:comparison of moduli stacks}
There is a natural open immersion
\begin{align}
  \cMrel \hookrightarrow \cMpot
\end{align}
of stacks.
\end{proposition}

\begin{proof}
The morphism given by sending a relation $I$
to the potential $\Phi$ defined by \eqref{eq:potential from relation}
is an isomorphism to its image;
the rational inverse is given by sending $\Phi$ to
$
\lb
\frac{\partial}{\partial a} \Phi
\rb_{
a \in \Qt_1 \setminus Q_1
}
$,
which is not defined on the locus
where
$
\frac{\partial}{\partial a} \Phi
= 0
$
for some
$
a \in \Qt_1 \setminus Q_1
$.
\end{proof}

\subsection{Toric GIT}

Let
\begin{align}
\gammat \colon \lb \Gm \rb^{\Qt_0} \to \lb \Gm \rb^{\Qt_1}
\end{align}
be the homomorphism
sending
$
\lb \lambda_v \rb_{v \in \Qt_0}
$
to
$
\lb \lambda_{t(a)} \lambda_{s(a)}^{-1} \rb_{a \in \Qt_1}
$,
and
\begin{align}
\deltat \colon \Gm \to \lb \Gm \rb^{\Qt_0}
\end{align}
be the diagonal homomorphism
sending
$
\beta
$
to
$
(\beta)_{v \in \Qt_0}
$.
We
set
$
K \coloneqq \image \rho
$
and
$
T \coloneqq \coker \rho
$,
so that
we have an exact sequence
\begin{align} \label{eq:KT}
1
\to K
\to \lb \Gm \rb^{\scrC_3 \lb \Qt \rb}
\xto{\varphi} T
\to 1
\end{align}
of algebraic groups.

\begin{lemma}
One has an exact sequence
\begin{align}
1
\to
\Gm
\xto{\deltat}
\lb \Gm \rb^{\Qt_0}
\xto{\gammat}
\lb \Gm \rb^{\Qt_1}
\xto{\rho}
K
\to
1
\end{align}
of algebraic groups.
\end{lemma}

\begin{proof}
It is straightforward to see
$
\gammat \circ \deltat = 1
$,
so that $\image \deltat \subset \ker \gammat$.
For any $(\lambda_v)_{v \in \Qt_0} \in \ker \gammat$,
one has $\lambda_v = \lambda_w$ if $v$ and $w$ are connected by an arrow $a \in \Qt_1$.
Since $\Qt$ is connected,
one obtains $\lambda_v = \lambda_w$ for all $v, w \in \Qt_0$,
and $\ker \gammat = \image \deltat$ is proved.

Similarly,
it is straightforward to see
$
\rho \circ \gammat = 1
$,
so that
$
\image \gammat \subset \ker \rho
$.
Let
$
\sfT \subset \Qt_1
$
be a spanning tree.
Any element
$
\alpha
= (\alpha_a)_{a \in \Qt_1}
\in \ker \rho
$
is determined by
$
(\alpha_a)_{a \in \sfT}
\in \lb \Gm \rb^{\sfT}
$,
and any element of
$
\lb \Gm \rb^{\sfT}
$
can be lifted to an element of
$
\ker \rho
$,
so that
the projection
$
\lb \Gm \rb^{\Qt_1}
\to
\lb \Gm \rb^{\sfT}
$
restricts to an isomorphism
$
\ker \rho
\simto
\lb \Gm \rb^{\sfT}
$.
It is clear that
the composite of $\gammat$
and the projection to
$
\lb \Gm \rb^{\sfT}
$
induces an isomorphism
$
\coim \deltat \simto \lb \Gm \rb^{\sfT}
$,
so that
$
\image \deltat = \ker \rho
$.
\end{proof}

It follows that
$K$
is isomorphic to
$
\lb \Gm \rb^{\left| \Qt_1 \right| - \left| \Qt_0 \right| + 1}
=
\lb \Gm \rb^{19}
$.
Set
\begin{align}\label{eq:rigidified moduli stack of potentials}
\cM
\coloneqq
\left[
\bA^{\scrC_3 \lb \Qt \rb}
\middle/
K
\right],
\end{align}
which is the rigidification of $\cMpot$
containing
$
\left[
(\Gm)^{\scrC_3 \lb \Qt \rb}
\middle/
K
\right]
\cong T
$
as an open substack.
By applying the functor
$
\Hom(\Gm, -)
$
to \eqref{eq:KT},
one obtains the exact sequence
\begin{align} \label{eq:fan sequence}
0
\to
L
\to
\bZ^{\scrC_3 \lb \Qt \rb}
\xto{\varphi_*}
N
\to
0
\end{align}
where
$
L \coloneqq \Hom(\Gm, K)
$
and
$
N \coloneqq \Hom(\Gm, T)
$.
The exact sequence
\begin{align} \label{eq:divisor sequence}
0
\to
M
\to
\bZ^{\scrC_3 \lb \Qt \rb}
\xto{\varphi^*}
L^\dual
\to
0
\end{align}
dual to
\pref{eq:fan sequence}
is obtained by
applying the functor
$
\Hom(-, \Gm)
$
to \eqref{eq:KT},
where
$
L^\dual \coloneqq \Hom(L, \bZ) \cong \Hom(\Gm, K)
$
and
$
M \coloneqq \Hom(T,\Gm) \cong N^\dual
$.
Let
$
\Sigma
$
be a fan
whose set of primitive generators of one-dimensional cones
is given by
\begin{align}
\Sigma ( 1 )
=
\lc \varphi_*(\bfe_c) \relmid c \in \scrC_3 \lb \Qt \rb \rc
\subset N,
\end{align}
where
$
\lc \bfe_c \rc_{c \in \scrC_3 \lb \Qt \rb}
$
is the standard basis of
$
\bZ^{\scrC_3 \lb \Qt \rb}
$.
The effective cone
$
A^+ \lb X_\Sigma \rb
$
of the toric variety
$
X_\Sigma
$
associated with $\Sigma$
inside the divisor class group
$
\Cl \lb X_\Sigma \rb \otimes \bR
$
can be identified with the cone
\begin{align}
A^+(\cM)
\coloneqq
\Cone \lc \varphi^*(\bfe_c) \relmid c \in \scrC_3 \lb \Qt \rb \rc
\subset
L^\dual \otimes \bR
\end{align}
under the isomorphism
$
\Cl \lb X_\Sigma \rb
\cong
L^\dual
$.

\begin{lemma} \label{lm:strongly convex}
The cone $A^+(\cM)$ is strongly convex.
\end{lemma}

\begin{proof}
This follows
\begin{align}
 \la \rho_*(\bsone), \varphi^*(\bfe_c) \ra = 3 > 0
\end{align}
for all
$
c \in \scrC_3 \lb \Qt \rb
$,
where
$
\bsone \coloneqq (1, \ldots, 1)
\in
\bZ^{\Qtilde_1}
= \Hom \lb \Gm \lb \Gm \rb^{\scrC_3 \lb \Qt \rb} \rb
$.
\end{proof}

For each
$
\chi
\in
L^\dual
$,
the categorical quotient
\begin{align}
\Mbar_\chi
\coloneqq
\lb \bA^{\scrC_3 \lb \Qt \rb} \rb^{\chiss}
\GIT K
\end{align}
of the $\chi$-semistable locus
is a toric variety
containing the geometric quotient
\begin{align}
M_\chi
\coloneqq
\lb \bA^{\scrC_3 \lb \Qt \rb} \rb^{\chis}
/ K
\end{align}
of the $\chi$-stable locus
as an open dense subvariety.
\pref{lm:strongly convex} implies the projectivity of
$
\Mbar_\chi
$
by \cite[Proposition 14.3.10]{MR2810322}.
The $\chi$-semistable locus
$
\lb \bA^{\scrC_3 \lb \Qt \rb} \rb^{\chiss}
$
is non-empty
if and only if
$
\chi
\in
A^+(\cM)
$
by
\cite[Proposition 14.3.5]{MR2810322}.
The $\chi$-stable locus
$
\lb \bA^{\scrC_3 \lb \Qt \rb} \rb^{\chis}
$
is non-empty
if and only if
$
\chi
$
is in the interior of
$
A^+(\cM)
$,
which is the case if and only if
$
\lb \bA^{\scrC_3 \lb \Qt \rb} \rb^{\chis}
$
contains
$
\lb \Gm \rb^{\scrC_3 \lb \Qt \rb}
$
by \cite[Proposition 14.3.6]{MR2810322}.
The fan
supported on $N \otimes \bR$
describing the toric variety $\Mbar_\chi$
will be denoted by
$\Sigma_\chi$.
The \emph{secondary fan}
is the fan
supported on
$
A^+(\cM)
$
consisting of the closures
of the equivalence classes
with respect to the equivalence relation
such that $\chi$ is defined to be equivalent to $\chi'$
if
$
\lb \bA^{\scrC_3 \lb \Qt \rb} \rb^{\chi\text{-ss}}
=
\lb \bA^{\scrC_3 \lb \Qt \rb} \rb^{\chi'\text{-ss}}
$
and
$
\lb \bA^{\scrC_3 \lb \Qt \rb} \rb^{\chi\text{-s}}
=
\lb \bA^{\scrC_3 \lb \Qt \rb} \rb^{\chi'\text{-s}}
$.
A maximal-dimensional cone
in the secondary fan
is called a \emph{chamber}.
A character $\chi$ is said to be \emph{generic}
if it is contained in the interior of a chamber,
which is the case
if and only if
$\Sigma_\chi$ is simplicial.

\section{Moduli of representations}
\label{sc:N}

Let
$
\module \bfk \Gamma
$
be the category of finitely generated modules
over the path algebra of the quiver with relations
$
\Gamma = (Q,I).
$
We assume
\(
   Q
\)
is finite and the underlying graph of
\(
   Q
\)
is connected. An isomorphism
\begin{align}
\dim
\colon
K \lb \module \bfk \Gamma \rb
\simto
\bZ^{Q_0}
\end{align}
from the Grothendieck group
is given by sending a module
to its \emph{dimension vector}
\begin{align}
\dim M
\coloneqq
\lb \dim_\bfk \lb M \bfe_v \rb \rb_{v \in Q_0}.
\end{align}
For
$
\bsd = (d_v)_{v \in Q_0}
\in \bZ^{Q_0}
$,
the moduli stack of $\bfk Q$-modules
with dimension vector $\bsd$
is the quotient stack
\begin{align} \label{eq:moduli stack of representations}
\cN \coloneqq
\ld
\prod_{a \in Q_1} \bA \lb \Hom(\bfk^{d_{t(a)}}, \bfk^{d_{s(a)}}) \rb^\dual
\middle/
\prod_{v \in Q_0} \GL(d_v)
\rd,
\end{align}
and the moduli stack
$
\cN(I)
$
of $\bfk \Gamma$-modules
is a closed substack of $\cN$
defined by the relations $I$.
We set
$
\bsd
=
( \rank E_v )_{v \in Q_0}
= \bsone \coloneqq (1, \ldots, 1),
$
so that
\begin{align}
\cN(Q) = \ld \bA^{Q_1} \middle/ (\Gm)^{Q_0} \rd,
\end{align}
where $(\Gm)^{Q_0}$ acts
on $\bA^{Q_1}$
through the morphism
$
\gamma \colon (\Gm)^{Q_0} \to (\Gm)^{Q_1}
$
presented by the signed incidence matrix.
Since the quiver \(Q\) is connected,
the kernel of $\gamma$ is the image of the diagonal embedding
$
\Gm \to (\Gm)^{Q_0}
$.
Hence
\(
    \sfK \coloneqq \image \gamma
    \cong
    \coim \gamma
    \cong
    \coker
    \left[
        \Gm \to (\Gm)^{Q_0}        
    \right]
\).
Set
$
\sfT \coloneqq \coker \gamma
$,
so that we have an exact sequence
\begin{align} \label{eq:KT2}
    1 \to \sfK \to \lb \Gm \rb^{Q_1} \xto{\phi} \sfT \to 1.
\end{align}
Note that
\(
   \sfT
\)
is isomorphic to
\(
   \Gm
   ^{
    | Q _{ 1 } |
    -
    | Q _{ 0 } |
    +
    1
   }
\),
so that~\eqref{eq:KT2} is split. By applying the functors
$
\Hom(\Gm, -)
$
and
$
\Hom(-,\Gm)
$
to \pref{eq:KT2},
one obtains the exact sequences
\begin{align}
0 \to \sfL \to \bZ^{Q_1} \xto{\phi_*} \sfN \to 0
\end{align}
and
\begin{align}
0 \to \sfM \to \bZ^{Q_1} \xto{\phi^*} \sfL^\dual \to 0
\end{align}
where
$
\sfN \coloneqq \Hom(\Gm, \sfT)
$,
$
\sfM \coloneqq \Hom(\sfT, \Gm)
$,
and
$
\sfL \coloneqq \Hom(\Gm, \sfK)
$.
A choice of a \emph{stability parameter}
\begin{align}
\theta
&\in \sfL^\dual
% \cong \Hom(K, \Gm) \\
\cong \lc (\theta_v)_{v \in Q_0} \in \bZ^{Q_0} \relmid \sum_{v \in Q_0} \theta_v = 0 \rc
\end{align}
gives a GIT quotient
\begin{align}\label{eq:GIT quotient}
\Nbar_\theta
\coloneqq
\left.
\lb \bA^{Q_1} \rb^{\theta\text{-ss}}
\middle/\!\!\!\middle/
\sfK
\right.
\end{align}
and its open subscheme
\begin{align}
N_\theta
\coloneqq
\left.
\lb \bA^{Q_1} \rb^{\theta\text{-s}}
\middle/
\sfK
\right.,
\end{align}
where
a $\bfk Q$-module $M$
of dimension vector $\bsone$
is \emph{$\theta$-stable}
(resp.~\emph{$\theta$-semistable})
if
$
\theta(N) > 0
$
(resp.~$
\theta(N) \ge 0
$)
for any non-zero submodule $N \subset M$
by \cite{King}.
A stability parameter $\theta$ is \emph{generic}
if
semi-stability
is equivalent to
stability.
From now on,
we will use the particular stability parameter
\begin{align} \label{eq:the special theta}
\theta
&=
\lb
\theta_{0,0}, \theta_{1,0}, \theta_{2,0},
\theta_{0,1}, \theta_{1,1}, \theta_{2,1},
\theta_{0,2}, \theta_{1,2}, \theta_{2,2}
\rb \\
&= \left(-11,-11,-11,3,3,6,7,7,7\right),
\end{align}
which is generic since $\sum_{v \in P} \theta_v \ne 0$
for any proper subset $P$ of $Q_0$.

\begin{lemma} \label{lm:N_theta(Q)}
The moduli scheme
$N_\theta$
is a projective toric variety
% of Picard number 8
containing $\sfT$ as the dense torus.
\end{lemma}

\begin{proof}
One has
\begin{align}
  \la \theta, \phi^*(\bfe_a) \ra > 0
\end{align}
for all $a \in Q_1$,
so that the cone
\begin{align}
A^+(\cN)
\coloneqq \Cone \lc \phi^* (\bfe_a) \in Q_1 \rc
\subset \sfL^\dual \otimes \bR
\end{align}
is strongly convex,
and hence $\Nbar_0$ is a point,
over which $\Nbar_\theta = N_\theta$ is projective.
Any point
in $(\Gm)^{Q_1}$
is $\theta$-stable,
since a submodule
of a $\bfk Q$-module
associated with a point
in $(\Gm)^{Q_1}$
corresponds bijectively to
a subset $S$ of $Q_0$ such that
if $a \in Q_1$ satisfies $t(a) \in S$, then $s(a) \in S$.
\end{proof}

The GIT quotient $N_\theta$ is a fine moduli scheme
carrying the universal representation
consisting of torus-invariant line bundles $\cU_v$
for $v \in Q_0$
and torus-invariant morphisms $\cU_{s(a)} \to \cU_{t(a)}$
for $a \in Q_1$.
We normalize the universal bundle
in such a way that
$\cU_{2,0} \cong \cO_{N_\theta}$.
The moduli scheme $N_\theta(I)$
of $\theta$-stable representations of $\Gamma$
is the closed subscheme of $N_\theta$
defined by $I$.
It is a fine moduli scheme again,
whose 
universal bundles,
which we write as $\cU_v$ again by abuse of notation,
are the restrictions
of the universal bundles
on $N_\theta$.

\begin{lemma}
\label{lm:N_theta}
The moduli scheme $N_\theta(I)$
for a general $[I] \in M_\chi$
is a
connected
and
reduced
curve
with the trivial dualizing sheaf.
\end{lemma}

\begin{proof}
One can easily see that the ideal
\begin{multline}\label{eq:toric relations}
  I _{ 0 }
  =
  \big(
    x_{0,1,0} x_{0,0,0},
    x_{1,1,0} x_{0,0,1},
    x_{2,1,0} x_{0,0,2},
    x_{1,1,2} x_{1,0,0},
    x_{2,1,2} x_{1,0,1}, \\
    x_{0,1,2} x_{1,0,2},
    x_{2,1,1} x_{2,0,0},
    x_{0,1,1} x_{2,0,1},
    x_{1,1,1} x_{2,0,2}
  \big),
\end{multline}
where
\(
   x
   _{
    i,
    j,
    k
   }
\)
corresponds to the arrow of the quiver from the vertex
\(
   (
    i,
    j
   )
\)
to
\(
   (
    i + k,
    j + 1
   )
\)
(the first entry should be taken modulo
\(
   3
\)),
defines a reduced complete intersection.
Since $N_\theta(I)$ constitutes a family of closed subschemes
of the projective toric variety $N_\theta$
parametrized by $I$,
it follows that $N_\theta$ is reduced
for a general $I$.
One can show
$
h^0 \lb \cO_{N_\theta(I_0)} \rb
= 1
$
by using the Macaulay 2 script \lstinline|connectedness.m2|,
which implies
that
$N_\theta(I)$ is connected
for a general $I$
by the upper semi-continuity
of the dimension of the cohomology
of a coherent sheaf. See~\pref{sc:connectedness} for an explanation of \lstinline|connectedness.m2|.

Since
$N_\theta(I_0)$
is a complete intersection
in the 10-dimensional projective toric variety
$N_\theta$,
the same is true for
$N_\theta(I)$
for general $I$
by the upper semi-continuity
of the dimension of fibers for a projective morphism.
Hence the dualizing sheaf of $N_\theta(I)$
is given
by the adjunction formula
as
\begin{align}
\omega_{N_\theta(I)}
\cong
\left.
\left(
  \omega_{N_\theta}
  \otimes
  \bigotimes_{i,j=0}^3
   \cU_{j,2} \otimes \cU_{i,0}^\dual
\right)
\right|_{N_\theta(I)},
\end{align}
which is trivial since
\begin{align}
\omega_{N_\theta}
  \otimes
  \bigotimes_{i,j=0}^3
   \cU_{j,2} \otimes \cU_{i,0}^\dual
\in
\Pic N_\theta
\cong
\sfL^\dual
\end{align}
is given by
\begin{align}
-\sum_{a \in Q_1} (\bfe_{t(a)} - \bfe_{s(a)})
+\sum_{i,j=0}^3 (\bfe_{j,2}-\bfe_{i,0})
= 0
\end{align}
where
$
\bfe_v
$
is a basis of $\bZ^{Q_0}$
associated with $v \in Q_0$.
\end{proof}

%
%-----------------------------------------------------
%

\section{Elliptic decuples}
\label{sc:Mell}

Let
$
Q = (Q_0, Q_1)
$
be the quiver in
\pref{fg:key quiver}.

\begin{definition} \label{df:admissible}
An \emph{elliptic decuple}
consists of
a smooth projective curve $Y$ of genus one
and
line bundles
$
\lb L_v \rb_{v \in Q_0}
$
satisfying
$
\deg L_{i,j} = j
$
for any
$
(i,j) \in Q_0
$.
An elliptic decuple is said to be
\emph{admissible}
if
$
\Hom \lb L _{i,k}, L _{j,k} \rb = 0
$
for any $i \ne j \in \{ 0, 1, 2 \}$
and any $k \in \{0, 1, 2 \}$.
\end{definition}

If
$
\lb
Y,
\lb L_v \rb_{v \in Q_0}
\rb
$
is an admissible elliptic decuple,
then one has
\begin{align}
\End \lb \bigoplus_{v \in Q_0} L_v \rb
\cong
\bfk Q / I
\end{align}
where the ideal $I$
satisfies \pref{eq:dimension matrix}.

\begin{definition} \label{df:moduli of elliptic decuples}
A family of admissible elliptic decuples over a scheme $S$
consists of a smooth projective morphism $\varphi \colon \cY \to S$
and a collection $(\cL_v)_{v \in Q_0}$ of line bundles
such that for any $\bfk$-valued point $s \in S(\bfk)$,
the fiber
$
\lb
Y_s \coloneqq \varphi^{-1}(s),
\lb \cL_v |_{Y_s} \rb_{v \in Q_0}
\rb
$
is an admissible elliptic decuple.
An isomorphism of families
$
\lb
\cY, (\cL_v)_{v \in Q_0}
\rb
$
and
$
\lb
\cY', (\cL_v')_{v \in Q_0}
\rb
$
of elliptic decuples
consists of an isomorphism
$
\psi \colon \cY \simto \cY'
$
of $S$-schemes,
a line bundle $\cL$ on $\cY$,
and an isomorphism
$
\psi_v \colon \cL_v \simto \psi^* \cL_v' \otimes \cL
$
of line bundles on $\cY$
for each $v \in Q_0$.
\end{definition}

\begin{proposition} \label{pr:decuple}
There exists a fine moduli scheme
$
\Mell
$
of admissible elliptic decuples
which is quasi-projective of dimension 8 over $\bfk$.
\end{proposition}

\begin{proof}
Note that for any invertible sheaf
$L$ of degree 1 on a nonsingular projective curve
\( Y \) of genus \( 1 \),
there exists a unique point \( y \in Y \) such that
\( L \simeq \cO _{ Y } ( y ) \).
Consider the forgetful morphism
$
F \colon \Mell \to M _{ 1, 3 }
$
to the fine moduli scheme of nonsingular 3-pointed curves of genus 1
sending
an elliptic decuple to the pointed curve
$
\left( Y, p _{ 0 }, p _{ 1 }, p _{ 2 } \right),
$
where
\(
  \left(p _{ i _{ 1 } }\right)
  _{
    i _{ 1 }
    \in
    \left\{
        0,1,2
    \right\}
  }
\)
are distinct points on \( Y \) such that
\(
  L _{ ( i _{ 1 }, 1 ) } \simeq \cO _{ Y } \left( p _{ i _{ 1 } } \right)
\).
Let
$
  \pi ^{ d } \colon \Pic ^{ d } \to M _{ 1, 3 }
$
be the relative Picard schemes parametrizing invertible sheaves of degree \( d \in \bZ \).
Note that the morphism \( \pi ^{ d } \) is projective for any \( d \in \bZ \).
Then there is an obvious open immersion
\begin{align}\label{eq:open immersion to what we know well}
  \Mell \hookrightarrow
  \Pic ^{ 0 } \times _{ M _{ 1, 3 } }
  \Pic ^{ 0 } \times _{ M _{ 1, 3 } }
  \Pic ^{ 2 } \times _{ M _{ 1, 3 } }
  \Pic ^{ 2 } \times _{ M _{ 1, 3 } }
  \Pic ^{ 2 }
\end{align}
to the nonsingular and irreducible quasi-projective scheme.
\end{proof}

\begin{lemma} \label{lm:decuple to spherical helix}
If
$
\lb
Y,
\lb L_v \rb_{v \in Q_0}
\rb
$
is an admissible elliptic decuple,
then
$
\lb L_v \rb_{v \in Q_0}
$
is a foundation of an acyclic spherical helix (see~\cite[Definition~4]{2007.07620} for the definition of acyclic spherical helices).
\end{lemma}

\begin{proof}
    It is a basic fact about a smooth projective curve
    \(
       Y
    \)
    of genus one and an invertible sheaf
    \(
       \cL
    \)
    on
    \(
       Y
    \)
    with
    \(
       \deg \cL \ge 0
    \)
    that
    \(
        H ^{ 1 }
        \left(
            Y,
            \cL
        \right)
        =
        0
    \)
    unless
    \(
       \cL
       \cong
       \cO _{ Y }
    \).
    On the other hand the spherical helix generated from the admissible elliptic decuple (see~\cite[Definition~3]{2007.07620} for the definition) satisfies
    \(
       \deg
       L _{ i, k }
       =
       k
    \)
    and that
    \(
       L _{ i, k }
       \not\cong
       L _{ j, k }
    \)
    for
    \(
       i
       \neq
       j
    \).
    The assertion follows from these facts.
\end{proof}

The main result of \cite{2007.07620}
combined with \pref{lm:decuple to spherical helix}
shows that an admissible elliptic decuple
gives an AS-regular $\tbz$-algebra of type $\Qt$.
The endomorphism algebra
$
\End \lb \bigoplus_{v \in Q_0} E_v \rb
$
of the resulting exceptional collection
is identified with
$
\End \lb \bigoplus_{v \in Q_0} L_v \rb
$
by construction.
This gives a morphism
$
\Mell \to \cMrel
$
of algebraic stacks,
and
let
$
\scrA \colon \Mell \dashrightarrow \Mbar_\chi
$
be the composite with the natural rational map
$
\cMrel \dashrightarrow \Mbar_\chi
$.

For
an admissible elliptic decuple
$
\lb
Y,
\lb L_v \rb_{v \in Q_0}
\rb
$,
let
$
\widetilde{c} \colon Y \to \cN(I)
$
be the classifying morphism
of the tautological family
$
\lb L_v \rb_{v \in Q_0}
$
of representations of $\bfk Q / I$
over $Y$,
where
$I$
is the ideal of relations
associated with the decuple.
This is an instance of a general construction
of the tautological morphism from a scheme
(or a stack)
to the moduli scheme (or the stack)
of modules
over the endomorphism algebra
of a coherent sheaf.
A sufficient condition for such a morphism
to be a closed immersion
is a generalization of the ampleness of a line bundle,
and
has been studied by many groups of people,
including \cite{Craw-Smith, Bergman-Proudfoot, MR3803802}
to name a few.

Recall our choice
\pref{eq:the special theta}
for the stability parameter.

\begin{lemma}
The classifying morphism
$
\widetilde{c}
$
factors through the open substack
of $\theta$-stable modules.
\end{lemma}

\begin{proof}
One needs to show that for any $y \in Y$,
the tautological representation $M$
does not have a submodule $N$
such that $\theta(N) > 0$.
It follows from the structure
of the zero loci of morphisms among
$
\lb L_v \rb_{v \in Q_0}
$
that
if the dimension vector $d = \dim N$
satisfies $d_{i,2} = 1$
for some $i \in \{ 0, 1, 2 \}$,
one has $d_{j,0} = 1$
for all $j \in \{ 0, 1, 2 \}$,
so that
$\theta(d) < 0$.
If $d_{i,2} = 0$
for all $i \in \{ 0, 1, 2 \}$
and $d_{j,1} = 1$
for some $j \in \{ 0, 1, 2 \}$,
then one has
$|\{ k \in \{ 0, 1, 2 \} \mid d_{k,0} = 1 \}| \ge 2$,
which implies $\theta(d) < 0$.
\end{proof}

Let
$
c \colon Y \to N_\theta(I)
$
be the morphism through which the morphism $\widetilde{c}$ factors.

\begin{lemma} \label{lm:closed immersion}
The classifying morphism $c$
is a closed immersion.
\end{lemma}

\begin{proof}
Choose paths
$
p _{ 1 }, p _{ 2 } \in \bfe_{0,2} \bfk \Gamma \bfe_{0,0}
$
and
$
q _{ 1 }, q _{ 2 } \in \bfe_{ 1, 2 } \bfk \Gamma \bfe_{ 0, 0 }
$
whose images,
denoted by the same symbols by abuse of notation,
generate
$
\Hom \left( L_{ 0, 0 }, L_{ 0, 2 } \right)
$
and
$
\Hom \left( L_{ 0, 0 }, L_{ 1, 2 } \right)
$
respectively.
The universal property of $N_\theta$
implies the existence of
$
\ptilde_i
\in
\Hom \lb \cU_{0,0}, \cU_{0,2} \rb
$
and
$
\qtilde_i
\in
\Hom \lb \cU_{0,0}, \cU_{1,2} \rb
$
satisfying
$
c^*(\ptilde_i) = p_i
$
and
$
c^*(\qtilde_i) = q_i
$
for $i \in \{ 1, 2 \}$.

Take an open subset
$
U \subset N_\theta(I)
$
containing the closed subset
$c(Y)$
and the invertible sheaf
$
\cU_{0,2} \otimes \cU_{1,2}
$
is generated by the global sections
$
  p _{ 1 } \otimes q _{ 1 },
  p _{ 1 } \otimes q _{ 2 },
  p _{ 2 } \otimes q _{ 1 }
$,
and
$
  p _{ 2 } \otimes q _{ 2 }
$.
Then we obtain the commutative diagram
\begin{equation}
\begin{tikzcd}
  Y \arrow[r, "c"] \arrow[rd, "i"'] & U \arrow[d, "\itilde"]\\
  & \bP ^{ 3 }
\end{tikzcd}
\end{equation}
where
\begin{align}
  i =\left[ p _{ 1 } \otimes q _{ 1 }:
  p _{ 1 } \otimes q _{ 2 }:
  p _{ 2 } \otimes q _{ 1 }:
  p _{ 2 } \otimes q _{ 2 } \right]
  \colon
  Y \to \bP ^{ 3 },\\
  \itilde =\left[ \ptilde _{ 1 } \otimes \qtilde _{ 1 }:
  \ptilde _{ 1 } \otimes \qtilde _{ 2 }:
  \ptilde _{ 2 } \otimes \qtilde _{ 1 }:
  \ptilde _{ 2 } \otimes \qtilde _{ 2 } \right]
  \colon
  U \to \bP ^{ 3 }.
\end{align}
The fact that
$i$ is a closed immersion
and
$\itilde$ is separated
(recall that
$
N_\theta
$
is projective over $\bfk$)
shows that
$c$ is also a closed immersion.
\end{proof}

\begin{lemma}\label{lm:structure of curve with trivial dualizing sheaf}
If
$Y$ and $Z$
are
connected
and reduced projective curves
with the trivial dualizing sheaves,
then any closed immersion
$
\iota \colon Y \hookrightarrow Z
$
is an isomorphism.
\end{lemma}

\begin{proof}
There is an isomorphism of \( \cO _{ Y } \)-modules
\begin{align}\label{eq:dualizing sheaf}
  \omega _{ Y }
  \cong \iota ^{ ! } \omega _{Z}
  =
  \cHom _{ Z} \left( \iota _{ * } \cO _{ Y }, \omega _{ Z } \right).
\end{align}
By the assumption, this is also isomorphic to
\(
  \cHom _{ Z} \left( \iota _{ * } \cO _{ Y }, \cO _{ Z } \right).
\)

Let
$
\cI \subset \cO_Z
$
be the ideal sheaf
of the closed subscheme
$
Y \subset Z
$.
By applying the functor
\(
  \cHom _{ Z} \left( \iota _{ * } \cO _{ Y }, - \right)
\)
to the short exact sequence
\begin{align}
  0 \to \cI \to \cO _{ Z } \to \iota _{ * } \cO _{ Y } \to 0,
\end{align}
one obtains an exact sequence
\begin{align}
  \begin{aligned}
  0
  \to
  \cHom _{ Z} \left( \iota _{ * } \cO _{ Y }, \cI \right)
  \to
  \cHom _{ Z} \left( \iota _{ * } \cO _{ Y }, \cO _{ Z} \right)
  \xrightarrow[]{ \alpha }
  \cHom _{ Z} \left( \iota _{ * } \cO _{ Y }, \iota _{ * } \cO _{ Y } \right)
  \cong
  \iota _{ * } \cO _{ Y }
\end{aligned}
\end{align}
of coherent sheaves on \( Z \).

Let
\( \cJ \subset \cO _{ Z }\)
be the ideal sheaf of the closed subscheme
\(
  \overline{ Z \setminus Y } \subset Z
\).
The reducedness of \( Z\) implies
\( \cI \cJ = 0\), so that
\(
  \cJ \subset \ann _{ \cO _{ Z }} \cI
\).
Since
\(
  \cI _{ y } = \cO _{ Z, y }
\)
for any closed point
\( y \in Z \setminus Y \),
one can also check the other inclusion, to conclude
\(
  \cJ = \ann _{ \cO _{ Z }} \cI
\).

Since the subsheaf
\(
  \cHom _{ Z} \left( \iota _{ * } \cO _{ Y }, \cI \right)
  \subset
  \omega _{ Y }
\)
is supported in the finite set
\(
  Y \cap \overline{ Z \setminus Y } \subset Y
\)
and
\(
  \omega _{ Y }
\)
is an invertible sheaf on \( Y \), it follows that
\(
  \cHom _{ Z} \left( \iota _{ * } \cO _{ Y }, \cI \right) = 0
\).
Therefore
\( \omega _{ Y } \) is identified with the ideal sheaf \( \image \alpha  \subset \cO _{ Y } \), which is obviously contained in
\( \ann _{ \cO _{ Z } } ( \cI ) \cdot \cO _{ Y } \).

If \( Z \neq Y \), then the connectedness of \( Z \) implies that
\(
  Y \cap \overline{ Z \setminus Y } \neq \emptyset
\),
which
in turn
coincides with the co-support of
\(
  \cJ \cdot \cO _{ Y }
  =
  \ann _{ \cO _{ Z } } ( \cI ) \cdot \cO _{ Y }
\).
Hence one has
\begin{align}
  \omega _{ Y }
  \subseteq
  \ann _{ \cO _{ Z } } ( \cI ) \cdot \cO _{ Y }
  \subsetneq
  \cO _{ Y },
\end{align}
which contradicts
\(
\omega _{ Y } \cong \cO _{ Y }
\)
since
\(
  \End _{ Y } ( \omega _{ Y } )
  \cong
  \bfk.
\)
Thus we conclude the proof.
\end{proof}

From Lemmas
\ref{lm:N_theta},
\ref{lm:closed immersion},
and
\ref{lm:structure of curve with trivial dualizing sheaf},
one obtains the following:

\begin{corollary} \label{cr:isomorphism}
The closed immersion
$
c
$
is an isomorphism,
and one has
$
L_v \cong c^* \cU_v
$
for any $v \in Q_0$.
\end{corollary}

\pref{th:birational}
is an immediate consequence of
\pref{cr:isomorphism}.

\begin{remark}
Since
$
\Mell
$
is birational to $M_{1,8}$,
\pref{th:birational}
gives a new proof
for the rationality of $M_{1,8}$,
which was originally shown in \cite{MR2716762}.
It is known by \cite{MR2716762}
that
$M_{1,n}$ is rational for
$n \le 10$,
and whether it can be proved
using moduli of marked noncommutative del Pezzo surfaces
of degree $11-n$
is an interesting question.
On the other hand,
$M_{1,n}$ is irrational for $n \ge 11$
since 
$
\kappa ( M _{ 1, 11} ) = 0
$
and
$
\kappa ( M _{ 1, n} ) = 1
$
for
$
n \ge 12
$
by \cite{MR2216264}.
\end{remark}
%
%-----------------------------------------------------
%

\section{Proof of \pref{th:immersion}}
\label{sc:immersion}

A configuration of six points on $\bP^2$ in general position
can be described by a matrix
of the form
\begin{align}\label{equation:the matrix p}
  p
  \coloneqq
  \begin{bmatrix}
    1 & 1 & 1 & 1 & 0 & 0\\
    a & c & 1 & 0 & 1 & 0\\
    b & d & 1 & 0 & 0 & 1
  \end{bmatrix}
\end{align}
for
$
a, b, c, d \in \Gm
$,
where each column is a homogeneous coordinate
of a point on $\bP^2$,
four of which are normalized by the action of $\Aut \bP^2 \cong \PGL_3(\bfk)$.
This gives a description of $\Mcubic$
as an open subscheme of $(\Gm)^4$.
We write the $(i,j)$-entry
of the matrix $p$
as $p_{ij}$,
the $j$-th column of $p$ as $\bsp_j$,
and
the column vector $(x,y,z)^T$
as $\bsp$.
For
distinct
$
i, j \in \left\{ 1, 2, 3, 4, 5, 6 \right\}
$,
the defining equation of the line
passing through $\bsp_i$ and $\bsp_j$
is given by
\begin{align}
\ell _{ i j } \coloneqq | \bsp _{i} \bsp _{ j } \bsp |,
\end{align}
and
for each
$
i \in \left\{ 1, 2, 3 \right\}
$,
the defining equation of the conic through the 5 points
$
\left\{
p _{ 1 }, \dots, p _{6}\right\} \setminus \left\{ p _{ i }
\right\}
$
is given by the following equation, where each index which appears on the right hand side should be replaced by the element of the set
\(
   \left\{
      1,
      2,
      3
   \right\}
\)
congruent to it modulo
\(
   3
\)
(say,
\(
   4
\)
should be replaced with
\(
   1
\)):
\begin{multline}
  q _{ i }
  \coloneqq
    p_{3,i+1}p_{3,i+2}
\begin{vmatrix}
p_{2,i+1} & p_{2,i+2}\\
p_{1,i+1} & p_{1,i+2}\\
\end{vmatrix}
xy+
p_{1,i+1}p_{1,i+2}
\begin{vmatrix}
p_{3,i+1} & p_{3,i+2}\\
p_{2,i+1} & p_{2,i+2}\\
\end{vmatrix}yz\\
+
p_{2,i+1}p_{2,i+2}
\begin{vmatrix}
p_{1,i+1} & p_{1,i+2}\\
p_{3,i+1} & p_{3,i+2}\\
\end{vmatrix}zx
\end{multline}
For each
$
j \in \{ 0, 1, 2 \}
$,
the relation
among the three paths
from $( 1, 0 )$
to $( j, 2 )$
is given by the linear dependence
\begin{align}
s \ell _{ 1 j + 4 } q _{ 1 } + t \ell _{ 2 j + 4 } q _{ 2 } + u \ell _{ 3 j + 4 } q _{ 3 } = 0,
\end{align}
where
$(s,t,u)$
is computed
by substituting
$
(x,y,z) = ( 0, p _{ 2, j + 2 }, p _{ 2, j } )
$
and
$
(x,y,z) = ( 0, p _{ 3, j + 2 }, p _{ 3, j } )
$
as
\begin{align}
(s,t,u) = ( p _{ 1, j + 1 }, p _{ 2, j + 1 }, p _{ 3, j + 1 } ),
\end{align}
where each second index which appears on the right hand sides should be replaced by the element of the set
\(
   \left\{
      4,
      5,
      6
   \right\}
\)
congruent to it modulo
\(
   3
\)
(say,
\(
   p _{ 2, 0 }
\)
should be replaced with
\(
   p _{ 2, 6 }
\)).
Similarly, for each
$
j \in \{ 0, 1, 2 \}
$,
the relation among the three paths
from $( 2, 0 )$
to $( j, 2 )$
is given by the linear dependence
\begin{align}
s \ell _{ 2 j + 4  } \ell _{ 3 1 } + t \ell _{ 3 j + 4  } \ell _{ 1 2 } + u \ell _{ 1 j + 4 } \ell _{ 2 3 } = 0,
\end{align}
where $(s,t,u)$ is computed
by substituting
$
\bsp = \bsp _{ 1 } + \bsp _{ 2 }
$
and
$
\bsp = \bsp _{ 2 } + \bsp _{ 3 }
$
as
\begin{align}
(s,t,u) = ( 1, 1, 1 ).
\end{align}
Finally,
the relation
among the three paths
from $(0,0)$
to $(j,2)$
is given by the linear dependence
\begin{align}
s \ell _{ 3 j + 4 } + t \ell _{ 1 j + 4 } + u \ell _{ 2 j + 4 } = 0,
\end{align}
where $(s,t,u)$ is computed
by substituting
$
\bsp = \bsp _{ 1 }
$
and
$
\bsp = \bsp _{ 2 }
$
as
\begin{align}
(s,t,u) =
\left(
  | \bsp _{ 1 } \bsp _{ j + 4 } \bsp _{ 2 } |,
  | \bsp _{ 2 } \bsp _{ j + 4 } \bsp _{ 3 } |,
  | \bsp _{ 3 } \bsp _{ j + 4 } \bsp _{ 1 } |
\right).
\end{align}
The assumption that the matrix $p$ of~\eqref{equation:the matrix p} is generic implies that
all of 27 coefficients of the 9 relations are non-zero,
so that the morphism
$
\Mcubic \to \cMpot
= \ld \bA^{\scrC_3 \lb \Qt \rb} \middle/ (\Gm)^{\Qt_1} \rd
$
sending a marked cubic surface
to the corresponding potential of the quiver
factors through the substack
$
\ld (\Gm)^{\scrC_3 \lb \Qt \rb} \middle/ (\Gm)^{\Qt_1} \rd
$,
whose coarse moduli scheme is $M^\circ$.
Since a cubic surface is recovered from the derived category of coherent sheaves
and its marking is recovered from a basis of the Grothendieck group,
a marked cubic surface is recovered from the potential,
and \pref{th:immersion} is proved.

\section{AS-regular algebras from noncommutative blow-up}
\label{sc:blow-up}

Let
$
  \left( Y, \sigma, \cO _{ Y } ( H ), p _{1 }, \dots, p _{ 6 } \right)
$
be
a nonuple
consisting of
\begin{itemize}
  \item 
a smooth projective curve
\(Y\)
of genus \(1\),
 \item 
a very ample invertible sheaf
\( \cO _{ Y } ( H ) \)
of degree \(3\) on \(Y\),
\item
a translation
\(\sigma\)
of \(Y\) of infinite order,
and
\item
six points
\(
  p _{1 }, \dots, p _{ 6 } \in Y
\)
in sufficiently general position.
\end{itemize}
In \cite[Chapter 12]{MR1846352},
a noncommutative graded ring \( F \)
associated with the nonuple
$
  \left( Y, \sigma, \cO _{ Y } ( H ), p _{1 }, \dots, p _{ 6 } \right)
$
is introduced.
The idea is that
\( F \) is a flat deformation of the anti-canonical ring
of commutative cubic surfaces,
so that
\(
    \Qcoh \Xtilde
    \coloneqq
    \Qgr F
\)
has the right to be called
a noncommutative cubic surface
obtained by blowing-up
the noncommutative projective plane $X$
associated with the triple $(Y, \sigma, \cO_Y(H))$
at the six points $p_1,\ldots,p_6$.

The graded algebra \( F \) admits a regular central element
\( t \in F _{ 1 } \)
of degree \( 1 \) such that
\begin{align}\label{eq:F / t F = R ( Y, N, tau )}
  F / t F \simeq R ( Y, \normalbundle, \tau ),
\end{align}
where the right hand side is the twisted homogeneous coordinate ring of \( Y \) associated to
\begin{align}
  \tau
  &
  = \sigma ^{ 3 },\\
  \normalbundle
  &
  =
  \cO _{ Y } ( H ) \otimes \sigma ^{ * } \cO _{ Y } ( H ) \otimes \left( \sigma ^{ 2 } \right) ^{ * } \cO _{ Y } ( H ) \otimes
  \cO _{ Y } \left( - \sum _{ i = 1 } ^{ 6 } p _{ i } \right).
\end{align}
The pair \( ( \normalbundle, \tau )\) is ample, so that there exists an equivalence of categories
\begin{align}\label{eq: Qgr F / t F = Qcoh Y}
  \Qgr F / t F \simeq \Qcoh Y.
\end{align}
Hence we obtain an adjoint pair
\begin{align}\label{eq:iota * iota*}
  \iota ^{ * } \colon
  \Qcoh \Xtilde
  =
  \Qgr F \rightleftarrows \Qcoh Y \colon \iota _{ * }
\end{align}
of functors.
The functor \( \iota _{ * } \) is exact, whereas \( \iota ^{ * } \) is only right exact.
These functors restrict to the subcategories
\( \coh \Xtilde = \qgr F \) and \( \coh Y \), and we let \( \bL \iota ^{ * }\) denote the left derived functor of \( \iota ^{ * }\).

It is also shown in \cite[Chapter 12]{MR1846352} that there exists a 4-dimensional AS-regular Koszul algebra
\begin{align}\label{eq:presentation of P}
  P = \bfk \langle x _{ 1 }, x _{ 2 }, x _{ 3 }, t \rangle /
  \left(
  r ' _{ 1 }, r ' _{ 2 }, r ' _{ 3 }, t x _{ 1 } - x _{ 1 } t, t x _{ 2 } - x _{ 2 } t,
  t x _{ 3 } - x _{ 3 } t
  \right),
\end{align}
where
\begin{align}\label{eq:defining relations of P}
  r ' _{ 1 }
  =
  c x _{ 1 } ^{ 2 } + a x _{ 2 } x _{ 3 } + b x _{ 3 } x _{ 2 }
  +
  \ell _{ 1 1 } x _{ 1 } t + \ell _{ 1 2 } x _{ 2 } t + \ell _{ 1 3 } x _{ 3 } t
  +
  \alpha _{ 1 } t ^{ 2 },\\
  r ' _{ 2 }
  =
  c x _{ 2 } ^{ 2 } + a x _{ 3 } x _{ 1 } + b x _{ 1 } x _{ 3 }
  +
  \ell _{ 2 1 } x _{ 1 } t + \ell _{ 2 2 } x _{ 2 } t + \ell _{ 2 3 } x _{ 3 } t
  +
  \alpha _{ 2 } t ^{ 2 },\\
  r ' _{ 3 }
  =
  c x _{ 3 } ^{ 2 } + a x _{ 1 } x _{ 2 } + b x _{ 2 } x _{ 1 }
  +
  \ell _{ 3 1 } x _{ 1 } t + \ell _{ 3 2 } x _{ 2 } t + \ell _{ 3 3 } x _{ 3 } t
  +
  \alpha _{ 3 } t ^{ 2 },
\end{align}
for some
\begin{align}
  \ell _{ i j }, \alpha _{ i } \in \bfk.
\end{align}
The assumption that \(Y\) is nonsingular implies that \( a b c \neq 0 \).
We will assume that \( \sigma \neq \id _{ Y } \), so that \( t \in P \) is the unique central element of degree 1 up to constant.

As shown in \cite[Chapter 12]{MR1846352},
the graded ring \( P \) admits a regular normalizing element
\(
  C ' _{ 3 } \in P _{ 3 }
\)
of degree 3 such that
\(
  P / C ' _{ 3 } P \simeq F.
\)
We use this fact to compute the Serre functor of the bounded derived category
\( D ^{ b } \qgr F \).

\begin{lemma}\label{lm:balanced dualizing complex of P}
The balanced dualizing complex of \( P \) is
\( P ( - 4 ) [ 4 ] \).
\end{lemma}

\begin{proof}
Since \( P \) is a 4-dimensional AS-regular Koszul algebra, by~\cite[Theorem~9.2]{MR1469646},
there exits an automorphism \(\varphi\) of the graded algebra \(P \) such that the \(P\)-bimodule
\(
 P ^{\varphi \varepsilon } (-4)[4]
\)
is the balanced dualizing complex over \(P\).
Here
\( \varepsilon \) is the automorphism of \( P \) which acts on the degree \(i\) part as multiplication by \( ( - 1) ^{ i } \).
We need to show that \(\varphi = \varepsilon\).

Let
\( \varphibar \colon P / t P \simto P / t P \)
be the automorphism of \( P/ t P \) induced by \( \varphi \).
By \cite[Corollary 9.3]{MR1469646}, one can confirm that the balanced dualizing complex of
\( P / t P \) is \( P / t P ( - 3 ) [ 3 ] \). In fact, one can check that the matrix \( Q \) which appears in the assertion of \cite[Corollary 9.3]{MR1469646} is \( I _{ 3 } \).
On the other hand, by \cite[Corollary 5.5 and Proposition 5.2]{Yekutieli_DCNGA}, the balanced dualizing complex of \( P / t P \) is isomorphic to \( P / t P ( - 3 ) ^{ \varphibar \varepsilon } [ 3 ] \). Thus we see that
\begin{align}
  \varphibar ( \xbar _{ i } ) = - \xbar _{ i }
\end{align}
for
\( i = 1, 2, 3 \).
Namely, there are constants
\( \lambda _{ i } \in \bfk \) such that
\( \varphi ( x _{ i } ) = - x _{ i } + \lambda _{ i } t \)
for
\( i = 1, 2, 3 \).
Also, since \( \varphi \) respects the center of \(P\), there is a non-zero constant \( \mu \in \bfk \) such that
\( \varphi ( t ) = \mu t \). Combining these observations with the presentation \eqref{eq:presentation of P}, it is easy to conclude that
\( \lambda _{ 1 } = \lambda _{ 2 } = \lambda _{ 3 } = 0 \).
Finally, if at least one of \( \ell _{ i j } \) is non-zero, one can conclude that
\( \mu = - 1\), so that
\( \varphi = \varepsilon \). Otherwise one can perturb \(P\) to those with non-trivial \( \ell _{ i j}\)s and by continuity conclude that \( \mu = - 1\). This perturbation argument is justified as follows.

Suppose that \( P \) is an algebra as above with \( \ell _{ i j } = 0 \) for all \( i, j \).
Let
\(
  P _{ \zeta }
\),
where \( \zeta \in \bfk \), 
be the algebra obtained by replacing one of the defining relations
\( r ' _{ 1 } \)
of \( P = P _{ 0 } \) by
\( r ' _{ 1 } + \zeta x _{ 1 } t \).
Technically speaking, let
\( \cP \)
be the sheaf of graded algebras over \( \bA ^{ 1 } = \Spec \bfk [ \zeta ] \)
defined as the quotient of the free algebra
\(
  \bfk \langle x _{ 1 }, x _{ 2 }, x _{ 3 }, t \rangle \otimes _{ \bfk } \cO _{ \bA ^{ 1 } }
\)
by the subsheaf of ideals generated by the global sections
\( r ' _{ 1 } + \zeta x _{ 1 } t, r ' _{ 2 }, r ' _{ 3 }, [ t, x_{1} ], [ t, x _{ 2 }], [ t, x _{ 3 } ] \).
One can easily confirm that the fiber of \( \cP \) over a \( \bfk \)-valued point \( \zeta \in \bA ^{ 1 } ( \bfk ) \) is \( P _{ \zeta } \).
By \cite[Theorem 3.1.3]{MR1429334},
\(
  P _{ \zeta }
\)
is an AS-regular Koszul algebra for any \( \zeta\). In fact, with the notation of \cite[Theorem 3.1.3]{MR1429334}, \(x ^{* } _{ j } = x _{j }\) in our case. Hence we can find the obvious solution \( \gamma _{ 1 } = \gamma _{ 2 } = \gamma _{ 3 } = 0 \) for the system of equations \cite[p.181, (3.2)]{MR1429334}.

Consider the Koszul dual \( \cP ^{ ! } \) of \( \cP \) over \( \bA ^{ 1 } \), which is a locally free sheaf of finite dimensional algebras such that the fiber of \( \cP ^{ ! } \) over \( \zeta \) is isomorphic to
\( \left(P _{ \zeta }\right) ^{ ! } \).
By \cite[Proposition 5.10]{MR1388568},
\( \cP ^{ ! } \)
is a sheaf of Frobenius algebras and hence there is an isomorphism of
\( \cP ^{ ! } \)-bimodules
\begin{align}
  \cHom _{ \bA ^{ 1 } } \left( \cP ^{ ! }, \cO _{ \bA ^{ 1 } } \right)
  \simeq
  \left(\cP ^{ ! }\right) ^{ \Phi ^{ ! } },
\end{align}
where \( \Phi ^{ ! } \) is an automorphism of \( \cP ^{ ! } \) which by functoriality corresponds to an automorphism \( \Phi \) of \( \cP \).

Similarly, \( \left(P _{ \zeta }\right) ^{ ! } \) is a Frobenius algebra and hence there is an isomorphism of \( \left(P _{ \zeta }\right) ^{ ! } \)-bimodules
\begin{align}
  \Hom _{ \bfk } \left( \left(P _{ \zeta }\right) ^{ ! }, \bfk \right)
  \simeq
  \left(\left(P _{ \zeta }\right) ^{ ! }\right) ^{ \left( \varphi _{ \zeta } \right) ^{ ! } }
\end{align}
for an automorphism \( \varphi _{ \zeta } \) of \( P _{ \zeta } \).
One can easily confirm that
\( \Phi | _{ \zeta } = \varphi _{ \zeta } \).

On the other hand, recall that we already know that the balanced dualizing bimodule of \( P _{ \zeta }\) for \( \zeta \neq 0\) is
\( P _{ \zeta } ( - 4 ) [ 4 ] \).
Hence by \cite[Theorem 9.2]{MR1469646}, we know \( \varphi _{ \zeta} = \varepsilon\)
 for all \( \zeta \neq 0 \).
 Thus we can conclude by continuity that
 \( \varphi _{ 0 } = \Phi | _{ 0 } = \varepsilon \).
 Again by \cite[Theorem 9.2]{MR1469646}, the automorphism \( \varphi \) of \( P \) is the same as \( \varphi _{ 0 } \). Thus we conclude the proof.
\end{proof}

\begin{lemma}\label{lm:Serre functor of the noncommutative blowup}
The functor
  \(
  (-1)[2]
  \)
is a Serre functor of the bounded derived category
\( D ^{ b } \qgr F \).
\end{lemma}

\begin{proof}
    We use~\cite[Corollary~5.5]{Yekutieli_DCNGA}.
    We take the surjective homomorphism
    \(
       P
       \to
       F
    \)
    for the homomorphism
    \(
       f \colon A \to B
    \)
    of~\cite[Proposition~5.2]{Yekutieli_DCNGA},
    \( P ( - 4) [ 4 ] \)
    for
    \(
       R ^{ \bullet }
    \)
    by
    \pref{lm:balanced dualizing complex of P},
    and
    \(
        F ( - 1 ) [ 3 ]
    \)
    for
    \(
       S ^{ \bullet }
    \).
    The conditions of \cite[Proposition~5.2]{Yekutieli_DCNGA} are checked as follows:
    \begin{align}
        \hom _{ P ^{ \op } } \left( F, P ( - 4 ) [ 4 ] \right)
        &
        \simeq
        \Cone \left( P ( - 4 ) [ 4 ] \xrightarrow[]{C' _{ 3 } \cdot } P ( - 1 ) [ 4 ] \right) [ - 1 ]\\
        &
        \simeq
        \Res _{ F \otimes P ^{ \op } } F ( - 1 ) [ 3 ],\\
        \hom _{ P } \left( F, P ( - 4 ) [ 4 ] \right)
        &
        \simeq
        \Cone \left( P ( - 4 ) [ 4 ] \xrightarrow[]{\cdot C' _{ 3 } } P ( - 1 ) [ 4 ] \right) [ - 1 ]\\
        &
        \simeq
        \Res _{ P \otimes F ^{ \op } } F ( - 1 ) [ 3 ].
    \end{align}
    Thus we have confirmed that
    \(
        F ( - 1 ) [ 3 ]
    \)
    is the balanced dualizing complex of
    \(
       F
    \).
    Hence by~\cite[Theorem~A.4]{MR2058456}, the functor
    \( ( - 1) [ 2 ] \)
    is a Serre functor of the bounded derived category
    \( D ^{ b } \qgr F \).
\end{proof}

\begin{proposition}\label{pr:restriction map is an isomorphism}
Let
\( ( \cE, \cF ) \)
be an exceptional pair of the derived category
\( D ^{ b } \qgr F \).
Then the restriction map
\begin{align}
  \hom _{ \qgr F } \left( \cE, \cF \right)
  \to
  \hom _{ \coh Y } \left( \bL \iota ^{ \ast } \cE, \bL \iota ^{ \ast } \cF \right)
\end{align}
is an isomorphism.
\end{proposition}

\begin{proof}
    As we saw in~\eqref{eq:F / t F = R ( Y, N, tau )}, we have the following short exact sequence of graded
    \(
       F
    \)-bimodules:
    \begin{align}
        0
        \to
        F ( - 1 )
        \xto{
            t \cdot
        }
        F
        \to
        \frac{
            F
        }{
            t F
        }
        \simeq
        R
        \left(
            Y,
            \normalbundle,
            \tau
        \right)
        \to
        0
    \end{align}
    Combining this with~\eqref{eq:iota * iota*},
    for any object \( \cF \in D ^{ b } \qgr F \) we obtain the following distinguished triangle:
    \begin{align}
        \cF ( - 1 )
        \to
        \cF
        \to
        \iota _{ * } \bL \iota ^{ * } \cF
        \xto{ + 1 }
    \end{align}
    Applying
    \( \hom _{ \Xt } \left( \cE, - \right) \)
    to this, we obtain the distinguished triangle as follows:
    \begin{align}
        \begin{aligned}
            \hom _{ \Xt } \left( \cE, \cF ( - 1 ) \right)
            \to
            \hom _{ \Xt } \left( \cE, \cF \right)
            \to
            \hom _{ \Xt } \left( \cE, \iota _{ * } \bL \iota ^{ * } \cF \right)
            \simeq
            \hom _{ Y } \left( \bL \iota ^{ \ast } \cE, \bL \iota ^{ \ast } \cF \right)
            \xto{+1}
        \end{aligned}
    \end{align}
    By \pref{lm:Serre functor of the noncommutative blowup}, we see that
    \begin{align}
        \hom _{ \Xt } \left( \cE, \cF ( - 1 ) \right)
        =
        \hom _{ \Xt } \left( \cE, \bS \left(\cF\right) [ - 2 ] \right)
        \simeq
        \hom _{ \Xt } \left( \cF, \cE \right) ^{ \vee } [ - 2 ]
        =
        0,
    \end{align}
    where
    \( \bS \)
    denotes the Serre functor of the bounded derived category \(D ^{ b } \coh \Xt\).
\end{proof}

By repeatedly using \cite[Theorem~8.4.1]{MR1846352} one obtains a semiorthogonal decomposition
\begin{align}\label{eq:Orlov type SOD}
    D ^{ b } \coh \Xtilde
    =
  \langle
  \cO _{ l _{ 1 } } ( - 1 ), \dots, \cO _{l _{ 6 }} ( - 1 ), D ^{ b } \coh X
  \rangle,
\end{align}
where
\(
  \cO _{ l _{ 1 } } ( - 1 ), \dots, \cO _{l _{ 6 }} ( - 1 )
\)
are exceptional objects.
Moreover,
\cite[p.~84]{MR1846352} gives the short exact sequence as follows.
\begin{align}
  0
  \to
  \cO _{ l _{ i } } ( - 1 )
  \to
  \cO _{ l _{ i } }
  \to
  \iota _{ * } \cO _{ \tau p _{ i } }
  \to
  0
\end{align}
Combining this with the isomorphism of functors
\begin{align}
  ( 1 ) \simeq \left(\normalbundle\right) _{ \tau }
  \coloneqq
  \tau _{ * } \left( - \otimes \normalbundle \right)
\end{align}
under the equivalence
\begin{align}
  \qgr R ( Y, \normalbundle, \tau )
  \simeq
  \coh Y,
\end{align}
we conclude that
\begin{align}
  \iota ^{ * } \cO _{ l _{ i } } ( -1 ) \simeq \cO _{ p _{ i } }.
\end{align}
Moreover we see that
\(
    \iota ^{ * } \cO _{ l _{ i } } ( -1 )
    \simeq
    \bL \iota ^{ * } \cO _{ l _{ i } } ( -1 )
\),
since the derived restriction \( \bL \iota ^{ * } \cO _{ l _{ i } } ( -1 ) \) of the exceptional object
\( \cO _{ l _{ i } } ( -1 ) \)
is a spherical object on \( Y \)
by the proof of \pref{pr:restriction map is an isomorphism},
and a spherical (and hence simple) object of $D^b \coh Y$
is a shift of a sheaf.

The collection
\pref{eq:key collection}
is obtained from
the Orlov type collection \cite{Orlov_PB}
\begin{align}\label{eq:Orlov type collection}
\tau_1
=
(E_i)_{i=1}^9
=
\lb
\cO _{ l _{ 1 } } ( - 1 ), \dots, \cO _{ l _{ 6 } } ( - 1 ),
\cO ( -2 l ), \cO ( - l ), \cO
\rb
\end{align}
by the following sequence of mutations:
\begin{align}
\tau_2
&= L_1 \circ L_2 \circ L_3 \circ L_4 \circ L_5 \circ L_6 (\tau_1) \\
&= \big( \cO(-2l+l_1+\cdots+l_6),\cO_{l_1}(-1),\ldots,\cO_{l_6}(-1), \cO(-l), \cO \big), \\
\tau_3
&= R_5 \circ R_6 \circ R_7 (\tau_2) \\
&= \big( \cO(-2l+l_1+\cdots+l_6),\cO_{l_1}(-1),\cO_{l_2}(-1), \cO_{l_3}(-1), \cO(-l)\\
&\qquad \qquad \qquad \qquad \cO(l_4-l), \cO(l_5-l), \cO(l_6-l), \cO \big), \\
\tau_4
&= R_6 \circ R_7 \circ R_8 (\tau_3) \\
&= \big( \cO(-2l+l_1+\cdots+l_6),\cO_{l_1}(-1),\cO_{l_2}(-1), \cO_{l_3}(-1), \cO(-l), \\
&\qquad \qquad \qquad \qquad \cO, \cO(l-l_4), \cO(l-l_5), \cO(l-l_6) \big), \\
\tau_5
&= L_2 \circ L_3 \circ L_4 (\tau_4) \\
&= \big( \cO(-2l+l_1+\cdots+l_6), \cO(-l+l_1+l_2+l_3), \cO_{l_1}(-1),\cO_{l_2}(-1), \cO_{l_3}(-1),\\
&\qquad \qquad \qquad \qquad \cO, \cO(l-l_4), \cO(l-l_5), \cO(l-l_6) \big), \\
\tau_6
&= R_3 \circ R_4 \circ R_5 (\tau_5) \\
&= \big( \cO(-2l+l_1+\cdots+l_6), \cO(-l+l_1+l_2+l_3), \cO, \cO(l_1), \cO(l_2), \cO(l_3), \\
&\qquad \qquad \qquad \qquad \cO(l-l_4), \cO(l-l_5), \cO(l-l_6) \big)
\\
\tau_7
&= L _{ 1 } \circ L _{ 2 } (\tau_6) \\
&= \big( \cO, \cO(-2l+l_1+\cdots+l_6), \cO(-l+l_1+l_2+l_3), \cO(l_1), \cO(l_2), \cO(l_3), \\
&\qquad \qquad \qquad \qquad \cO(l-l_4), \cO(l-l_5), \cO(l-l_6) \big).
\end{align}

The noncommutative projective plane
\(
   X
\)
admits a semiorthogonal decomposition of the form
\begin{align}
    D ^{ b } \coh X
    =
   \langle
     \cO ( - 2 l ),
     \cO ( - l ),
     \cO
   \rangle,
\end{align}
which together with~\eqref{eq:Orlov type SOD} yields a full exceptional collection of
\(
   D ^{ b } \coh \Xt
\)
of the form~\eqref{eq:Orlov type collection}. By abuse of notation we let
\(
   \tau _{ 1 }
\)
denote thus-obtained full exceptional collection of
\(
   D ^{ b } \coh \Xt
\),
and
\(
   \tau _{ i }
\)
denote those obtained from
\(
   \tau _{ 1 }
\)
by the same sequence of mutations as in the commutative case.

Note that the restriction of $\tau_1$ gives the spherical collection
\begin{align}
\sigma_1
=
(\cO_{p_1}, \ldots, \cO_{p_6}, L_0^{-5} L_1^3, L_0^{-2} L_1, \cO)
\end{align}
in $D^b \coh E$,
where
\(
   L _{ 0 }
   =
   \cO _{ Y } ( H )
\)
and
\(
   L _{ 1 }
   =
   \sigma ^{ \ast }
   L _{ 0 }
\).
Since mutation commutes with restriction by~\pref{pr:restriction map is an isomorphism}, the restrictions $\sigma_k$ of $\tau_k$ are given by
\begin{align}
\sigma_2
&= \big(
L_0^{-5} L_1^3(p_1+\cdots+p_6),
\cO_{p_1}, \ldots, \cO_{p_6}, L_0^{-2} L_1, \cO
\big),
\end{align}
\begin{multline}
\sigma_3
= \big(
L_0^{-5} L_1^3(p_1+\cdots+p_6),
\cO_{p_1}, \cO_{p_2}, \cO_{p_3}, L_0^{-2} L_1, \\
L_0^{-2} L_1(p_4), L_0^{-2} L_1(p_5), L_0^{-2} L_1(p_6), \cO
\big),
\end{multline}
\begin{multline}
\sigma_4
= \big(
L_0^{-5} L_1^3(p_1+\cdots+p_6),
\cO_{p_1}, \cO_{p_2}, \cO_{p_3}, L_0^{-2} L_1, \\
\cO, L_0^{2} L_1^{-1}(-p_4), L_0^{2} L_1^{-1}(-p_5), L_0^{2} L_1^{-1}(-p_6)
\big),
\end{multline}
\begin{multline}
\sigma_5
= \big(
L_0^{-5} L_1^3(p_1+\cdots+p_6),
L_0^{-2} L_1(p_1+p_2+p_3),
\cO_{p_1}, \cO_{p_2}, \cO_{p_3}, \\
\cO, L_0^{2} L_1^{-1}(-p_4), L_0^{2} L_1^{-1}(-p_5), L_0^{2} L_1^{-1}(-p_6)
\big),
\end{multline}
\begin{multline}
\sigma_6
= \big(
L_0^{-5} L_1^3(p_1+\cdots+p_6),
L_0^{-2} L_1(p_1+p_2+p_3),
\cO,
\cO(p_1), \cO(p_2), \cO(p_3), \\
L_0^{2} L_1^{-1}(-p_4), L_0^{2} L_1^{-1}(-p_5), L_0^{2} L_1^{-1}(-p_6)
\big)
\end{multline}
\begin{multline}
\sigma_7
= \big(
    \cO,
    L_0^{-5} L_1^3(p_1+\cdots+p_6),
    L_0^{-2} L_1(p_1+p_2+p_3),
    \cO(p_1), \cO(p_2), \cO(p_3), \\
    L_0^{2} L_1^{-1}(-p_4), L_0^{2} L_1^{-1}(-p_5), L_0^{2} L_1^{-1}(-p_6)
\big).\label{eq:sigma7}
\end{multline}

Putting everything together, we obtain the following:
\begin{theorem}
    One has a derived equivalence
    \begin{align} \label{eq:Db qgr A=Db nc blow-up}
        D ^{ b } \coh \Xt
        \simeq
        D ^{ b } \qgr A
    \end{align}
   where
    \(
       A
    \)
    is the AS-regular
    \(
       \left\{
          0,
          1,
          2
       \right\}
       \times
       \bZ
    \)-algebra associated with the elliptic decuple
    $(Y, \sigma_7)$.
\end{theorem}

\begin{proof}
Let $\Gamma = (Q, I)$ be the quiver in \pref{fg:key quiver}
with the relation $I$ determined by the elliptic decuple.
The equivalence \pref{eq:Db qgr A=Db nc blow-up}
is the composite of
the equivalence
\begin{align}
   D^b \coh \Xt \simeq D^b \module \bfk \Gamma
\end{align}
coming from the full strong exceptional collection
$\tau_7$,
and
the equivalence
\begin{align}
D^b \module \bfk \Gamma
\simeq
D^b \qgr A
\end{align}
coming from the full strong exceptional collection
\begin{align}
\lb
\cO(0,0), \cO(1,0), \cO(2,0),
\cO(0,1), \cO(1,1), \cO(2,1),
\cO(0,2), \cO(1,2), \cO(2,2)
\rb
\end{align}
obtained by
\cite[Theorem 16.(i)]{Orlov2009},
where
$\cO(i,j)$ is the image of the projective $A$-module $P_{i,j}$
by the natural functor $\module A \to \qgr A$.
\end{proof}

%
%------------------------------------------------------------------------
%
\appendix

%
%------------------------------------------------------------------------
%
\section{Connectedness of the moduli of representations}
\label{sc:connectedness}

We explain how to check that the moduli space
\(
   N
   _{
    \theta
   }
   ( I _{ 0 } )
\)
of stable representations for the generic stability condition
\(
   \theta
\)
as in~\eqref{eq:the special theta} and the toric relations
\(
   I _{ 0 }
\)
as in~\eqref{eq:toric relations} satisfies
\(
   h ^{ 0 }
   \left(
    N
    _{
     \theta
    }
    ( I _{ 0 } ),
    \cO
    _{
        N
        _{
         \theta
        }
        ( I _{ 0 } )
    }
   \right)
   =
   1
\)
by using the Macaulay2 script \lstinline|connectedness.m2|
in the ancillary file to this paper.

As we saw in~\pref{sc:N}, the moduli space
\(
   N
   _{
    \theta
   }
   ( I _{ 0 } )
\)
is isomorphic to the closed subscheme of the GIT quotient~\eqref{eq:GIT quotient} defined by the ideal
\(
   I _{ 0 }
\)
in a suitable fashion. The
\(
   18
   \times
   9
\)
matrix of weights of the action
\(
   \Gm
   ^{
    Q _{ 0 }
   }
   \curvearrowright
   \bA
   ^{
    Q _{ 1 }
   }
\)
is given in the beginning of \lstinline|connectedness.m2| as follows.

\begin{lstlisting}
    Weight = {
        {-1,0,0,1,0,0,0,0,0},{-1,0,0,0,1,0,0,0,0},{-1,0,0,0,0,1,0,0,0},
        {0,-1,0,0,1,0,0,0,0},{0,-1,0,0,0,1,0,0,0},{0,-1,0,1,0,0,0,0,0},
        {0,0,-1,0,0,1,0,0,0},{0,0,-1,1,0,0,0,0,0},{0,0,-1,0,1,0,0,0,0},
        {0,0,0,-1,0,0,1,0,0},{0,0,0,-1,0,0,0,1,0},{0,0,0,-1,0,0,0,0,1},
        {0,0,0,0,-1,0,0,1,0},{0,0,0,0,-1,0,0,0,1},{0,0,0,0,-1,0,1,0,0},
        {0,0,0,0,0,-1,0,0,1},{0,0,0,0,0,-1,1,0,0},{0,0,0,0,0,-1,0,1,0}
        };
    
    WeightMatrix = matrix(Weight);    
\end{lstlisting}

First we compute the defining ideal of the
\(
   \theta
\)-unstable locus in
\(
   \bA ^{ Q _{ 1 } }
\).
For this purpose we use a Hilbert-Mumford criterion for stability (cf., e.g., \cite[Proposition~6.3]{MR3466868}).
% \todo{there should be more appropriate reference, but we might want to compromise on this.}
Recall that we have no strictly semistable points due to the genericity of
\(
   \theta
\).
In what follows we identify
\(
   Q _{ 1 }
\)
with the set of standard coordinates of the affine space
\(
   \bA
   ^{
    Q _{ 1 }
   }
\).

\begin{lemma}\label{lm:Hilbert-Mumford}
    Given a point of
    \(
       \bA ^{ Q _{ 1 } }
    \),
    let
    \(
       S
       \subseteq
       Q _{ 1 }
    \)
    be the set of coordinates which are non-zero at the point. Then the point is semistable for a character
    \(
       \chi
    \)
    if and only if the cone spanned by the weights of the coordinates in
    \(
       S
    \)
    contains
    \(
       \chi
    \).
    Moreover, it is stable if and only if the interior of the cone contains
    \(
       \chi
    \).
\end{lemma}

Recall that the weights of the coordinates and the character
\(
   \theta
\)
are all contained in the 8-dimensional lattice
\(
   \sfL ^{ \dual }
\).
By definition the defining ideal of the
\(
   \theta
\)-unstable locus in the coordinate ring
\(
   \bfk
   \left[
    \bA
    ^{
        Q _{ 1 }
    }
   \right]
\)
is generated by
\(
   \theta
\)-semi-invariants; note that such a semi-invariant is a linear combination of products of coordinates such that the weight of each product is a positive multiple of
\(
   \theta
\).
Combining this with~\pref{lm:Hilbert-Mumford} and the genericity of
\(
   \theta
\),
we can conclude that the defining ideal of the unstable locus is, up to possibly taking the radical, generated by products of octuples of coordinates satisfying the condition that the cone spanned by their weights contains
\(
   \theta
\).

The following part of \lstinline|connectedness.m2| computes the list \lstinline|Bases| of octuples of coordinates such that the cone spanned by their weights has non-empty interior, and then the sublist \lstinline|RelevantCones| by the condition that the character
\(
   \theta
\),
which is denoted by \lstinline|p| below, is contained in the cone. Then we turn it into the set \lstinline|Generators| of monomials and then let \lstinline|IrrelevantIdeal| be the radical of the ideal generated by \lstinline|Generators|. This is the defining ideal of the
\(
   \theta
\)-unstable locus.

\begin{lstlisting}
    Eights = subsets(for i in toList(0..17) list i,8);

    Bases = select(
        Eights,
        S -> rank (submatrix(transpose WeightMatrix, S)) == 8
        );

    loadPackage "Polyhedra";

    p = transpose matrix{{-11,-11,-11,3,3,6,7,7,7}};

    RelevantCones = select(
        Bases,
        S -> contains(
            posHull submatrix(transpose WeightMatrix, S),
            p
        )
        );

    R = QQ[x_0 .. x_17,Degrees=>Weight];

    Generators = for S in RelevantCones list(
        product(
                for j from 0 to 7 list(
                    x_(S_j)
                )
            )
    );

    IrrelevantIdeal = radical monomialIdeal Generators;
\end{lstlisting}

Note that the moduli space
\(
   N
   _{
    \theta
   }
   ( I _{ 0 } )
\)
is (the base change to \(\bfk\) of) a good quotient of
\(
    V
    \left(
        \mbox{\lstinline|DefiningIdeal|}
    \right)
    \cap
    \left(
        \bA
        ^{
         Q _{ 1 }
        }
    \right)
    ^{
        \thetass
    }
    \hookrightarrow
   \bA
   ^{
    Q _{ 1 }
   }
   =
   \Spec
   \bQ
   \left[
    x _{ 0 },
    \dots,
    x _{ 17 }
   \right]
\),
where the monomial ideal \lstinline|DefiningIdeal| is defined as follows.

\begin{lstlisting}
    DefiningIdeal = ideal(
    x_0*x_9,
    x_1*x_12,
    x_2*x_15,
    x_3*x_14,
    x_4*x_17,
    x_5*x_11,
    x_6*x_16,
    x_7*x_10,
    x_8*x_13
    );
\end{lstlisting}
As the closed subscheme
\(
    V
    \left(
        \mbox{\lstinline|DefiningIdeal|}
    \right)
    \hookrightarrow
   \bA
   ^{
    Q _{ 1 }
   }
\)
is the union of coordinate linear subspaces, the assertion
\(
   h ^{ 0 }
   \left(
    N
    _{
     \theta
    }
    ( I _{ 0 } ),
    \cO
    _{
        N
        _{
         \theta
        }
        ( I _{ 0 } )
    }
   \right)
   =
   1
\)
is equivalent to the connectedness of the intersection
\(
    V
    \left(
        \mbox{\lstinline|DefiningIdeal|}
    \right)
    \cap
    \left(
        \bA
        ^{
         Q _{ 1 }
        }
    \right)
    ^{
        \thetass
    }
\).

The following lines compute the set \lstinline|Primes| of minimal prime ideals of \lstinline|DefiningIdeal| and then its subset \lstinline|IrreducibleComponents| of prime ideals
\(
   P
\)
which does not contain the defining ideal \lstinline|IrrelevantIdeal| of the unstable locus. Note that \lstinline|IrreducibleComponents| is in bijection with the set of irreducible components of
\(
    N
    _{
     \theta
    }
    ( I _{ 0 } )
\).

\begin{lstlisting}
    Primes = minimalPrimes DefiningIdeal;

    IrreducibleComponents = select(Primes, P -> not isSubset(IrrelevantIdeal,P));
\end{lstlisting}
There are 18 elements in \lstinline|IrreducibleComponents|.

To conclude the proof we show that the incidence graph of the irreducible components of the moduli space
\(
    N
    _{
     \theta
    }
    ( I _{ 0 } )
\)
is connected. Concretely, we define the simple graph
\lstinline|g|
whose set of vertices is
\lstinline|IrreducibleComponents|
and the vertices corresponding to
\(
    P,
    P '
    \in
\)
\lstinline|IrreducibleComponents| are connected by an edge if and only if
\(
    P + P '
\)
does not contain
\lstinline|IrrelevantIdeal|.
In \lstinline|connectedness.m2| we let
\lstinline|AdjacentComponents| be the list of pairs
\(
   \left(
    P,
    P '
   \right)
\)
like this. We then turn it into the data \lstinline|Edges| of edges and then turn it into the graph \lstinline|g| by the command \lstinline|graph|.
Finally the command \lstinline|isConnected| applied to \lstinline|g| returns the value \lstinline|true|, showing the connectedness of the graph
\lstinline|g|.

\begin{lstlisting}
    AdjacentComponents = apply(
    toList(0..17),
    i-> select(
            toList(0..17),
            j -> (j != i) and not isSubset(
                IrrelevantIdeal,
                radical monomialIdeal(IrreducibleComponents#i+IrreducibleComponents#j)
                )
        )
    );
    
    needsPackage "Graphs";

    Edges' = apply(
        toList(0..17),
        i -> for j in {0,1} list(
                {
                    i,
                    (AdjacentComponents#i)#j
                }
            )
        );
    Edges = flatten Edges';

    g = graph(Edges);

    print isConnected g
\end{lstlisting}

%
%-------------------- back matter --------------------
%
\bibliographystyle{alpha}
\bibliography{mainbibs}

\begin{thebibliography}{LBSVdB96}

\bibitem[ACT02]{MR1910264}
Daniel Allcock, James~A. Carlson, and Domingo Toledo.
\newblock The complex hyperbolic geometry of the moduli space of cubic surfaces.
\newblock {\em J. Algebraic Geom.}, 11(4):659--724, 2002.

\bibitem[AOU]{abdelgadir2014compact}
Tarig Abdelgadir, Shinnosuke Okawa, and Kazushi Ueda.
\newblock Compact moduli of noncommutative projective planes.
\newblock arXiv:1411.7770 (2014).

\bibitem[AS87]{Artin_Schelter}
Michael Artin and William~F. Schelter.
\newblock Graded algebras of global dimension $3$.
\newblock {\em Adv. in Math}, 66(2):171--216, 1987.

\bibitem[AZ94]{Artin-Zhang_NPS}
M.~Artin and J.~J. Zhang.
\newblock Noncommutative projective schemes.
\newblock {\em Adv. Math.}, 109(2):228--287, 1994.

\bibitem[Bel98]{MR2716762}
Pavel Belorousski.
\newblock {\em Chow rings of moduli spaces of pointed elliptic curves}.
\newblock ProQuest LLC, Ann Arbor, MI, 1998.
\newblock Thesis (Ph.D.)--The University of Chicago.

\bibitem[BF06]{MR2216264}
Gilberto Bini and Claudio Fontanari.
\newblock Moduli of curves and spin structures via algebraic geometry.
\newblock {\em Trans. Amer. Math. Soc.}, 358(7):3207--3217, 2006.

\bibitem[BKR01]{Bridgeland-King-Reid}
Tom Bridgeland, Alastair King, and Miles Reid.
\newblock The {M}c{K}ay correspondence as an equivalence of derived categories.
\newblock {\em J. Amer. Math. Soc.}, 14(3):535--554 (electronic), 2001.

\bibitem[Bon89]{Bondal_RAACS}
A.~I. Bondal.
\newblock Representations of associative algebras and coherent sheaves.
\newblock {\em Izv. Akad. Nauk SSSR Ser. Mat.}, 53(1):25--44, 1989.

\bibitem[BP93]{MR1230966}
A.~I. Bondal and A.~E. Polishchuk.
\newblock Homological properties of associative algebras: the method of helices.
\newblock {\em Izv. Ross. Akad. Nauk Ser. Mat.}, 57(2):3--50, 1993.

\bibitem[BP08]{Bergman-Proudfoot}
Aaron Bergman and Nicholas~J. Proudfoot.
\newblock Moduli spaces for {B}ondal quivers.
\newblock {\em Pacific J. Math.}, 237(2):201--221, 2008.

\bibitem[CIK18]{MR3803802}
Alastair Craw, Yukari Ito, and Joseph Karmazyn.
\newblock Multigraded linear series and recollement.
\newblock {\em Math. Z.}, 289(1-2):535--565, 2018.

\bibitem[CLS11]{MR2810322}
David~A. Cox, John~B. Little, and Henry~K. Schenck.
\newblock {\em Toric varieties}, volume 124 of {\em Graduate Studies in Mathematics}.
\newblock American Mathematical Society, Providence, RI, 2011.

\bibitem[CS08]{Craw-Smith}
Alastair Craw and Gregory~G. Smith.
\newblock Projective toric varieties as fine moduli spaces of quiver representations.
\newblock {\em Amer. J. Math.}, 130(6):1509--1534, 2008.

\bibitem[dNvdB04]{MR2058456}
Koen de~Naeghel and Michel van~den Bergh.
\newblock Ideal classes of three-dimensional {S}klyanin algebras.
\newblock {\em J. Algebra}, 276(2):515--551, 2004.

\bibitem[Dol12]{MR2964027}
Igor~V. Dolgachev.
\newblock {\em Classical algebraic geometry}.
\newblock Cambridge University Press, Cambridge, 2012.
\newblock A modern view.

\bibitem[DvGK05]{MR2196731}
I.~Dolgachev, B.~van Geemen, and S.~Kond{\=o}.
\newblock A complex ball uniformization of the moduli space of cubic surfaces via periods of {$K3$} surfaces.
\newblock {\em J. Reine Angew. Math.}, 588:99--148, 2005.

\bibitem[Gin06]{Ginzburg_CYA}
Victor Ginzburg.
\newblock {C}alabi-{Y}au algebras.
\newblock math.AG/0612139, 2006.

\bibitem[HIO14]{MR3144232}
Martin Herschend, Osamu Iyama, and Steffen Oppermann.
\newblock {$n$}-representation infinite algebras.
\newblock {\em Adv. Math.}, 252:292--342, 2014.

\bibitem[HKT09]{MR2534095}
Paul Hacking, Sean Keel, and Jenia Tevelev.
\newblock Stable pair, tropical, and log canonical compactifications of moduli spaces of del {P}ezzo surfaces.
\newblock {\em Invent. Math.}, 178(1):173--227, 2009.

\bibitem[Kel94]{Keller_DDC}
Bernhard Keller.
\newblock Deriving {DG} categories.
\newblock {\em Ann. Sci. \'Ecole Norm. Sup. (4)}, 27(1):63--102, 1994.

\bibitem[Kel03]{keller2003derived}
Bernhard Keller.
\newblock Derived invariance of higher structures on the {H}ochschild complex.
\newblock available at \url{https://webusers.imj-prg.fr/~bernhard.keller/publ/dih.pdf}, 2003.

\bibitem[Kel11]{Keller_DCYC}
Bernhard Keller.
\newblock Deformed {C}alabi-{Y}au completions.
\newblock {\em J. Reine Angew. Math.}, 654:125--180, 2011.
\newblock With an appendix by Michel Van den Bergh.

\bibitem[Kin94]{King}
A.~D. King.
\newblock Moduli of representations of finite-dimensional algebras.
\newblock {\em Quart. J. Math. Oxford Ser. (2)}, 45(180):515--530, 1994.

\bibitem[KN98]{MR1642152}
B.~V. Karpov and D.~Yu. Nogin.
\newblock Three-block exceptional sets on del {P}ezzo surfaces.
\newblock {\em Izv. Ross. Akad. Nauk Ser. Mat.}, 62(3):3--38, 1998.

\bibitem[KV00]{Kapranov-Vasserot}
M.~Kapranov and E.~Vasserot.
\newblock Kleinian singularities, derived categories and {H}all algebras.
\newblock {\em Math. Ann.}, 316(3):565--576, 2000.

\bibitem[LBSVdB96]{MR1429334}
Lieven Le~Bruyn, S.~P. Smith, and Michel Van~den Bergh.
\newblock Central extensions of three-dimensional {A}rtin-{S}chelter regular algebras.
\newblock {\em Math. Z.}, 222(2):171--212, 1996.

\bibitem[Min12]{MR2874928}
Hiroyuki Minamoto.
\newblock Ampleness of two-sided tilting complexes.
\newblock {\em Int. Math. Res. Not. IMRN}, (1):67--101, 2012.

\bibitem[MM11]{MR2770441}
Hiroyuki Minamoto and Izuru Mori.
\newblock The structure of {AS}-{G}orenstein algebras.
\newblock {\em Adv. Math.}, 226(5):4061--4095, 2011.

\bibitem[Nar82]{MR662660}
Isao Naruki.
\newblock Cross ratio variety as a moduli space of cubic surfaces.
\newblock {\em Proc. London Math. Soc. (3)}, 45(1):1--30, 1982.
\newblock With an appendix by Eduard Looijenga.

\bibitem[Oka16]{MR3466868}
Shinnosuke Okawa.
\newblock On images of {M}ori dream spaces.
\newblock {\em Math. Ann.}, 364(3-4):1315--1342, 2016.

\bibitem[Orl92]{Orlov_PB}
D.~O. Orlov.
\newblock Projective bundles, monoidal transformations, and derived categories of coherent sheaves.
\newblock {\em Izv. Ross. Akad. Nauk Ser. Mat.}, 56(4):852--862, 1992.

\bibitem[Orl09]{Orlov2009}
Dmitri Orlov.
\newblock {\em Derived Categories of Coherent Sheaves and Triangulated Categories of Singularities}, pages 503--531.
\newblock Birkh{\"a}user Boston, Boston, 2009.

\bibitem[OU]{2007.07620}
Shinnosuke Okawa and Kazushi Ueda.
\newblock {AS}-regular algebras from acyclic spherical helices.
\newblock arXiv:2007.07620.

\bibitem[Ric89]{Rickard}
Jeremy Rickard.
\newblock Morita theory for derived categories.
\newblock {\em J. London Math. Soc. (2)}, 39(3):436--456, 1989.

\bibitem[Smi96]{MR1388568}
S.~Paul Smith.
\newblock Some finite-dimensional algebras related to elliptic curves.
\newblock In {\em Representation theory of algebras and related topics ({M}exico {C}ity, 1994)}, volume~19 of {\em CMS Conf. Proc.}, pages 315--348. Amer. Math. Soc., Providence, RI, 1996.

\bibitem[Ued05]{Ueda_SMQCCS}
Kazushi Ueda.
\newblock Stokes matrix for the quantum cohomology of cubic surfaces.
\newblock math.AG/0505350, 2005.

\bibitem[UY11]{Ueda-Yamazaki_NBTMQ}
Kazushi Ueda and Masahito Yamazaki.
\newblock A note on dimer models and {M}c{K}ay quivers.
\newblock {\em Comm. Math. Phys.}, 301(3):723--747, 2011.

\bibitem[UY13]{MR3100950}
Kazushi Ueda and Masahito Yamazaki.
\newblock Homological mirror symmetry for toric orbifolds of toric del {P}ezzo surfaces.
\newblock {\em J. Reine Angew. Math.}, 680:1--22, 2013.

\bibitem[VdB97]{MR1469646}
Michel Van~den Bergh.
\newblock Existence theorems for dualizing complexes over non-commutative graded and filtered rings.
\newblock {\em J. Algebra}, 195(2):662--679, 1997.

\bibitem[VdB01]{MR1846352}
Michel Van~den Bergh.
\newblock Blowing up of non-commutative smooth surfaces.
\newblock {\em Mem. Amer. Math. Soc.}, 154(734):x+140, 2001.

\bibitem[Yek92]{Yekutieli_DCNGA}
Amnon Yekutieli.
\newblock Dualizing complexes over noncommutative graded algebras.
\newblock {\em J. Algebra}, 153(1):41--84, 1992.

\end{thebibliography}

\end{document}